\documentclass[11pt,a4paper,reqno]{amsart}
\usepackage[applemac]{inputenc}
\usepackage[T1]{fontenc}
\usepackage{amsmath}
\usepackage{amsthm}
\usepackage{amsfonts}
\usepackage{amssymb}
\usepackage{graphicx}
\usepackage{color}
\usepackage{xcolor}
\usepackage{cancel}
\usepackage{amsbsy}
\usepackage{mathrsfs}
\usepackage{bm}
\usepackage{hyperref}
\usepackage{mathtools}
\usepackage{graphicx}
\usepackage{caption}
\usepackage{subcaption}
\usepackage{enumerate}
\usepackage[shortlabels]{enumitem}
\newtheorem{theorem}{Theorem}[section]
\newtheorem{lemma}[theorem]{Lemma}

\newtheorem{proposition}[theorem]{Proposition}
\newtheorem{corollary}[theorem]{Corollary}
\newtheorem{definition}[theorem]{Definition}
\theoremstyle{remark}
\newtheorem{remark}[theorem]{\it \bf{Remark}\/}
\newenvironment{acknowledgement}{\noindent{\bf Acknowledgement.~}}{}
\numberwithin{equation}{section}
\catcode`@=11
\def\section{\@startsection{section}{1}%
  \z@{1.5\linespacing\@plus\linespacing}{.5\linespacing}%
  {\normalfont\bfseries\large\centering}}
\catcode`@=12
\newcommand{\be}{\begin{equation}}
	\newcommand{\ee}{\end{equation}}
\newcommand{\bea}{\begin{eqnarray}}
	\newcommand{\eea}{\end{eqnarray}}
\newcommand{\bee}{\begin{eqnarray*}}
	\newcommand{\eee}{\end{eqnarray*}}

\def\pa{\partial}

\def\na{\nabla}

\def\CC{\mathbb{C}}

\def\RR{\mathbb{R}}
\def\ZZ{\mathbb{Z}}

\def\om{\omega}

\def\ep{\varepsilon}

\def\calB{{\mathcal B}}
\def\calA{{\mathcal A}}
\def\calD{{\mathcal D}}

\def\calF{{\mathcal F}}
\def\calG{{\mathcal G}}

\def\calK{{\mathcal K}}
\def\calL{{\mathcal L}}

\def\calQ{{\mathcal Q}}

\def\calX{{\mathcal X}}
\def\calY{{\mathcal Y}}

\def\bmrmd{{\bm{\mathrm{d}}}}

\def\calX{{\mathcal X}}

\def\calS{{\mathcal S}}

\catcode`@=11
\def\supess{\mathop{\operator@font Sup\,ess}}
\catcode`@=12

\def\CC{\mathbb{C}}

\def\RR{\mathbb{R}}

\def\ZZ{\mathbb{Z}}

\def\ZZ{\mathbb{Z}}

\def\a{\alpha}

\def\e{\varepsilon}

\def\bar#1{{\overline #1}}

\def\R2+{\RR ^2_+}

\def\pa{\partial}
\def\na{\nabla}

\def\lim{\mathop{\rm lim}}

\def\supp{{\rm supp}~}
\def\sup{\mathop{\rm sup}}

\def\l{\lambda}

\def\log{{\rm log}}

\def\pa{\partial}

\def\pa{\partial}
\def\la{\langle}

\def\ra{\rangle}

\def\L{\mathcal L}

\def\wl{\omega_{\mathrm{Lamb}}}
\def\psilm{\psi_{\text{Lamb}}^{\text{moving}}}

\def\psil{\psi_{\text{Lamb}}}
\def\errl{\text{Err}_{\text{Lamb}}}

\newcommand{\scl}[2]{\left(#1\,,\,#2\right)} 
\newcommand{\da}[1]{\left|#1 \right|} 
\newcommand{\dr}[1]{\left(#1\right)} 
\newcommand{\one}{\mathbf{1}}
\label{sec 2}

\begin{document}

\title[]{On the stability of Lamb-Chaplygin dipole for the 2D Euler equation}

 \author[Z. Li]{Zexing Li}
 \address{Laboratoire Analyse, G\'eom\'etrie et Mod\'elisation,
 CY Cergy Paris Universit\'e,
 2 avenue Adolphe Chauvin, 95300, Pontoise, France}
 \email{zexing.li@cyu.fr}

\author[P. Song]{Peicong Song}
 \address{Applied and Computational Mathematics, California Institute of Technology, Pasadena, California 91125, USA}
\email{psong2@caltech.edu}

\author[T. Zhou]{Tao Zhou}
 \address{Department of Mathematics\\
National University of Singapore\\
Singapore\\
119076\\
Singapore}
\email{zhoutao@u.nus.edu}

\maketitle

\begin{abstract}

The Lamb-Chaplygin dipole is a traveling wave solution to the 2D incompressible Euler equation, whose orbital stability was established in \cite{abe2022stability, abe2025stability} assuming the odd symmetry in $x_2$ (O) and non-negativity in upper half-plane (N). This paper is devoted to further study of its stability in the following two aspects. Firstly, we prove the spectral stability of the linearized operator around the Lamb-Chaplygin dipole without conditions (O) or (N), based on the index theory established in \cite{LinZeng22}. This excludes an instability mechanism by unstable eigenmodes, and provides rigorous evidence towards nonlinear stability in this general setting. Secondly, assuming (O) and (N), we refine the orbital stability results in \cite{abe2022stability, abe2025stability} quantitatively by proving a linear bound of the fluctuation and a uniform control of the moving velocity. Instead of using a variational approach, our proof relies on the construction of a new coercive Lyapunov functional with a delicate mixed structure: it is quadratic in the interior region, but linear in the exterior region.
\end{abstract}

\section{Introduction}

\subsection{2D Euler equation and Lamb-Chaplygin dipole}
In this paper, we are concerned with the 2D incompressible Euler equation in vorticity form,
\be
\partial_t \omega + \na^\perp  \Delta^{-1} \omega \cdot \na \omega=0,
\label{eq: Euler}
\ee
where $\na^\perp = (-\partial_{x_2}, \partial_{x_1})$ and $\Delta^{-1} \omega = \left( \frac{1}{2 \pi}\log |\cdot| \right) * \omega $. 

Global well-posedness for the 2D Euler equation \eqref{eq: Euler} is classical. In particular, for initial vorticity $\omega_0 \in H^s(\RR^2)$ with $s>2$, global well-posedness follows from the Beale--Kato--Majda criterion \cite{zbMATH03915851} together with the conservation of $\|\omega(t,\cdot)\|_{L^\infty}$; see \cite{vicalbook,zbMATH02547738,zbMATH03011177}. On the other hand, if $\omega_0 \in L^1(\RR^2)\cap L^\infty(\RR^2)$, then global well-posedness is also ensured by Yudovich's theory \cite{Yudovich1963}.

Additionally, there are several invariant quantities for 2D Euler equation \eqref{eq: Euler}, and we list the following three quantities which will be widely used in our paper.

\noindent $\bullet$ Kinetic energy:
\[
E[\omega] = -\frac{1}{4} \int_{\RR^2} \omega(\mathbf{x}) \psi(\mathbf x) d \mathbf x.
\]
Since the stream function is given by
\[
\psi(\mathbf x):= \Delta_{\RR^2}^{-1} \omega(\mathbf x)  = \frac{1}{2 \pi} \int_{\RR^2} \log \left( |\mathbf x - \mathbf y| \right) \omega(\mathbf y ) d\mathbf y,
\]
the kinetic energy can be written as
\be
E[\omega] = - \frac{1}{8 \pi} \iint_{\RR^2 \times \RR^2} \log \left( |\mathbf x - \mathbf y| \right) \omega (\mathbf x) \omega(\mathbf y) d \mathbf x d \mathbf y.
\label{kinetic energy}
\ee

\noindent $\bullet$ Enstrophy:
\be
K[\omega] := \frac{1}{2} \| \omega \|_{L^2(\RR^2)}^2.
\label{enstrophy}
\ee

\noindent $\bullet$ Impulse:
\be
I[\omega] := \frac{1}{2} \int_{\RR^2} x_2 \omega(\mathbf x) d \mathbf x.
\label{impulse: def}
\ee

Notably, the 2D Euler equation \eqref{eq: Euler} admits an explicit traveling wave solution $\wl(\mathbf x - t \mathbf{e}_1)$ called Lamb-Chaplygin dipole, whose profile $\wl$ is compactly supported on $B(0,1)$ and odd in $x_2$, i.e., $\wl(x_1,x_2) = -\wl(x_1,-x_2)$, which can be explicitly represented by
\be
\wl (\mathbf x)
= 
\begin{cases}
     -\frac{2c_L J_1(c_L r) }{J_0(c_L)} \sin \theta, & \text{for } r < 1, \\
     0, & \text{for } r \ge 1,
\end{cases}
\label{Lamb dipole: vorticity}
\ee
where $c_L$ is the first positive zero of the first order Bessel function of the first kind $J_1(r)$ \footnote{In the later discussion, we denote $J_m(r)$ the $m$-th order Bessel function of the first kind.}. In particular, the related kinetic energy, enstrophy and impulse are respectively given by
\be
E[\wl] = \pi, \quad K[\wl] = \pi c_L^2 \quad \text{ and } \quad  I[\wl] = \pi.
\label{energy+impulse+enstrophy of lamb dipole}
\ee
Additionally, if we denote the related stream function in the original coordinate by  $\psil := \Delta_{\RR^2}^{-1} \wl$, then the stream function in the moving frame defined by
\be
\psilm(\mathbf x) := \psil (\mathbf x) + x_2, \quad \text{for all} \; \mathbf x \in \RR^2,
\label{moving frame: stream function}
\ee
has the explicit form
\be
\psilm(\mathbf x) = 
\begin{cases}
     \frac{2 J_1(c_L r) }{c_L J_0(c_L)} \sin \theta, & \text{for } r < 1, \\
     \left(r - \frac{1}{r} \right) \sin \theta = \left( 1 - r^{-2} \right) x_2, & \text{for } r \ge 1.
\end{cases}
\label{Lamb dipole: stream function}
\ee
In particular, inside of the unit disk $B$, $\psilm$ and $\wl$ satisfy the following algebraic relation
\be
\wl (\mathbf x) = - c_L^2 \psilm(\mathbf x), \quad \text{for all} \; \mathbf{x} \in B,
\label{lamb dipole: algebraic relation}
\ee
while outside the disk $\psilm$ has definite sign (on the upper half-plane). This structure motivates our definition of the following error function: 
\begin{align}
\begin{split}
\errl(\mathbf x) & := \wl(\mathbf{x}) + c_L^2 \psil(\mathbf{x}) + c_L^2 x_2\\
&= c_L^2 \mathbf{1}_{B^c}(\mathbf x) \psilm(\mathbf x) \ge 0, \quad \text{for all} \mathbf \; \mathbf x \in \RR_+^2,
\label{error: def}
\end{split}
\end{align}
where we denote $\RR^2_+:=\{(x_1,x_2) \in \RR^2:x_2>0\}$.

\begin{figure}[htbp]
	\begin{center}
	\includegraphics[scale=1]{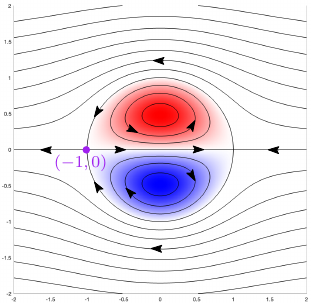}
	\caption{\label{fig_lamb} Illustration of the Lamb-Chaplygin dipole $\wl$ and its streamlines in the moving frame with velocity $1$.}
	\end{center}
\end{figure}

\subsection{Background on Lamb-Chaplygin dipole.}

\subsubsection{Lamb-Chaplygin dipole and its stability}\label{sec12}

\mbox{}

\vspace{0.1cm}

The Lamb-Chaplygin dipole has a long history and a lasting impact on the study of $2$D hydrodynamics. In 1906, H. Lamb noted a simple explicit traveling wave solution to \eqref{eq: Euler} in \cite{lamb1906hydrodynamics}, which turned out to be a special case of the non-symmetric Chaplygin dipoles independently discovered by  S. A. Chaplygin in 1903 in \cite{chaplygin1903} (see also \cite{chaplygin_1994}). The symmetric (with respect to the $x_2$-variable) dipoles are now referred to as the \textit{Lamb-Chaplygin dipoles} or the \textit{Lamb dipoles}. 

The dipole vortices, which include the Chaplygin-Lamb dipole as a special case, are considered as stable vortex structures in a certain sense, and there are many related experimental and numerical investigations \cite{afanasyev2006_experiments,heijst1994experiments,heijst1998numerics,Krasny_2021_numerics,Nielsen_numerics1997,protas2024linear,heijst1989dipole_experiments}. On the other hand, their mathematical stability is not fully understood.

Regarding the stability of Lamb-Chaplygin dipole, Abe and Choi \cite{abe2022stability} first rigorously proved the orbital stability of the Lamb-Chaplygin dipole under the following constraints of initial data: 
\begin{enumerate}[(a)]
    \item[(O)] \textit{Odd-symmetry in} $x_2$: $\om_0(x_1,x_2) = -\om_0(x_1,-x_2),$
    \item[(N)] \textit{Non-negativity in upper half-plane}: $\om_0(x_1,x_2)\geq 0$ for $x_2 \ge 0$,
    \item[(F)] \textit{Finite mass}: $\om_0\in L^1(\RR^2)$.
\end{enumerate}
They exploited the variational characterization of Lamb-Chaplygin dipole that $\wl$ is the maximizer of the kinetic energy $E[\omega]$ in the admissible class
\[
\calA_{\text{Lamb}} := \Big\{ \omega \in L^1 \cap L^\infty:  \text{$\omega$ odd in $x_2$},\; \omega \ge 0 \text{ on } \RR_+^2, \; I[\omega] = \pi \text{ and } K[\omega] =\pi c_L^2 \Big\}.
\]
In another work \cite[Theorem 5.1]{Wang_ConcentratedVortexPair24}, Wang obtained a similar orbital stability result of the Lamb-Chaplygin dipole using the results in \cite{burton2005isoperimetric,burton2021compactness}, with basically the same constraints
(O), (N), and (F),
but with slightly different norm for measuring stability.

However, removing the constraints (O), (N) and (F) seems to be extremely difficult. Very recently, there was progress made by Abe-Choi-Jeong \cite{abe2025stability}, and the authors removed the finite mass constraint (F).
Their proof is based on a new variational characterization of the Lamb-Chaplygin dipole that $\wl$ is the unique optimizer of the following energy interpolation inequality,
\be 
E[\omega] \le C_* \| x_2 \omega \|_{L^1(\RR_+^2)} \| \omega \|_{L^2(\RR_+^2)},
\label{energy ineq}
\ee 
for some $C_* > 0$ assuming $\omega$ odd in $x_2$, which was originally established in \cite[Corollary 2.5]{abe2025stabilitymultiplelambdipoles}.

Besides, for the Euler equation suitably restricted in a disk, we have a truncated Lamb dipole as a steady state, whose orbital stability without assumption (O) and (N) was established by Wang \cite{Wang_StabilityOnDisk25}. Nevertheless, the orbital stability of the Lamb-Chaplygin dipole on $\RR^2$ without (O) or (N) remains widely open.

When there are multiple Lamb-Chaplygin dipoles traveling in the same direction, orbital stability can also be established, provided the dipoles are sufficiently
separated and that the faster dipoles are positioned in front of the slower ones \cite{abe2025stabilitymultiplelambdipoles}. See also the orbital stability of two Lamb-Chaplygin dipoles traveling in opposite directions which are initially far away from each other \cite{choi2024stabilityvortexquadrupolesoddodd}.

For further results on other vortex dipoles for 2D Euler equations and vortex rings for 3D axisymmetric Euler equations, interested readers are referred to \cite{zbMATH06226652,zbMATH08159583,zbMATH07996976,zbMATH07782026,zbMATH08084253, choi2025vortex, choi2025existence, davila2023global,zbMATH08147947}.

\subsubsection{Applications of Lamb-Chaplygin dipole}

\mbox{}

\vspace{0.1cm}

The Lamb-Chaplygin dipole is a good candidate for constructing small-scale solutions to the 2D Euler equation. Precisely, Jeong and Choi \cite{zbMATH07488949} studied the filamentation near the Lamb-Chaplygin dipole and proved linear-in-time filamentation for arbitrarily small and localized perturbations, thus verifying the linear growth of vorticity gradient. Furthermore, since $(-1,0)$ is the saddle point of Lamb-Chaplygin dipole in the moving frame (see Figure \ref{fig_lamb}), together with the orbital stability established in \cite{abe2022stability, abe2025stability}, Jeong, Yao and the third author \cite{jeong2025superlinear} recently obtained the superlinear gradient growth for 2D Euler equation on $\RR^2$.

Furthermore, Lamb-Chaplygin dipole can also be applied to the ill-posedness for 2D Euler equation. In recent work \cite{brue2024flexibilitytwodimensionaleulerflows}, Bru\'{e}-Colombo-Kumar used the Lamb-Chaplygin dipole as a building block in the proof of the flexibility and nonuniqueness of $L^\infty L^p$-weak solutions (in terms of vorticity) for some $p>1$ to the $2$D Euler equations, using the convex integration scheme.

\subsection{Main results}

\subsubsection{On spectral stability of the Lamb-Chaplygin dipole}

\mbox{}

\vspace{0.1cm}
In the existing literature, all previous stability results for the Lamb-Chaplygin dipole rely crucially on the nonnegativity and odd-symmetry assumptions. To the best of our knowledge, no result has been established so far on the stability or instability of the Lamb-Chaplygin dipole once these two restrictions are removed. 
Our first result in this paper partially addresses this issue by studying the spectral stability of Lamb-Chaplygin dipole in more general functional classes (in particular, without symmetry or sign restrictions).

Before we present the main theorem, let us introduce the related linearized operator around the Lamb-Chaplygin dipole under 2D Euler equation \eqref{eq: Euler}. Precisely, if we take the ansatz $\omega(t, \cdot) = \left(\wl + h \right)(t, \cdot- t \mathbf{e}_1)$, then the residue $h$ solves the following equation:
\[
\partial_t h  + \calL h + \na h \cdot \na ^\perp \Delta^{-1} h =0,
\]
with the linear operator $\calL$ given by
\be
\calL h = - \mathbf{e}_1 \cdot \na h + \na h \cdot \na^\perp \Delta^{-1} \wl + \na \wl  \cdot \na^\perp \Delta^{-1} h.
\label{linearized op: def}
\ee
We denote by $\mathcal{L}: D(\calL) \subset L^2(\RR^2) \to L^2(\RR^2)$ the realization of the formal expression $\calL$ on $L^2(\RR^2)$ equipped with its maximal domain $D(\calL)$ defined by
\be \label{domain calL L2}
D\left( \calL \right)
:=
\Bigl\{
h\in L^2(\mathbb R^2)\;:\;
\calL h \in L^2(\mathbb R^2)
\text{ in the distributional sense}
\Bigr\}.
\ee
Besides, for any $\alpha>0$, define the weighted $L^2$-space as
\begin{equation}\label{definition of weighted L2 spaces X_alpha}
    X_\alpha:=\left\{f\in L^2(\RR^2):\int_{\RR^2}|f( \mathbf x)|^2 \la \mathbf x \ra ^{2 \alpha} \,dx<+\infty\right\} \text{ with } \la \mathbf x \ra := \sqrt{1+|\mathbf x|^2}.
\end{equation}
Then, we denote by $\calL_\a : D\left(\calL_\a \right) \subset X_\a \to X_\alpha$ the realization of the formal expression $\mathcal{L}$ on $X_\alpha$, defined on the maximal domain
\[
D\left(\calL_\a \right)
:=
\Bigl\{
h\in X_\alpha\;:\;
\calL h \in X_\alpha
\text{ in the distributional sense}
\Bigr\}.
\]
Now we are ready to introduce our first main result:
\begin{theorem}[Spectral stability of Lamb-Chaplygin dipole]
\label{main thm: linear}
For the linearized operators defined above, the following holds:\\
(\romannumeral 1) Mode stability in $L^2(\RR^2)$: The operator $\calL: D\left( \calL \right) \subset L^2(\RR^2)\to L^2(\RR^2)$ has the mode stability in $L^2(\RR^2)$ in the sense that it has no unstable eigenvalue. That is,
\[
\sigma_p \left(\calL \right)\cap \{\lambda\in\mathbb C:\Re\lambda<0\}=\varnothing.
\]
(\romannumeral 2) Spectral stability in $X_\alpha$: for any $\alpha>0$, $\calL_\alpha: D(\calL_\alpha) \subset X_\alpha\to X_\alpha$ is spectrally stable. That is,
\[
\sigma(\calL_\alpha)\cap \{\lambda\in\mathbb C:\Re\lambda<0\}=\varnothing.
\]

\end{theorem}

\mbox{}

\noindent \textit{Comments on Theorem \ref{main thm: linear}.}

\mbox{}

\noindent \textit{1. Towards the stability of the Lamb-Chaplygin dipole without conditions.}

\vspace{0.2cm}

As discussed in the Subsection \ref{sec12}, the stability of the Lamb-Chaplygin dipole in $\RR^2$ beyond odd-symmetry and non-negativity constraints is widely open and unclear. To the best of our knowledge, Theorem \ref{main thm: linear} provides the first rigorous investigation of this problem from a spectral perspective. This result excludes one important instability mechanism by unstable eigenmodes, as in the Taylor-Green vortices \cite{cao2026instability} and Kelvin-Stuart vortices \cite{liao2023stability}; and provides rigorous evidence towards nonlinear stability in this general setting.

\mbox{}

\noindent \textit{2. Role of underlying space, and comparison with the numerical investigation \cite{protas2024linear}.}

\vspace{0.2cm}

The realization of $\calL$ is sensitive with respect to the underlying space. Due to the singularity of Biot-Savart law at low frequency, one can show $\calL$ is not closable with domain \eqref{domain calL L2}, and thereafter $\sigma(\calL) = \CC$\footnote{The similar issues about non-closability and $\sigma(\calL) = \CC$ also appear in $L^p(\RR^2)$ sense with $p > 2$.}. That restrains our discussion of $\calL$ only for the \textit{mode stability} in $L^2(\RR^2)$ sense. In contrast, for functions with slightly more integrability (or decay) near infinity, $\calL_\a$ becomes a closed operator so that we can discuss its \textit{spectral stability} in $X_\a$ for any $\a > 0$. 

We also mention the numerical evidence \cite{protas2024linear} for the existence of unstable eigenmodes of $\calL$, which seems incompatible with our result. Notice that \cite[Figure 4]{protas2024linear} also indicates a large region of $\sigma_{\text{ess}}(\calL)\cap \{\lambda \in \CC: \Re \l < 0 \}$, but the paper did not specify the underlying space and domain of $\calL$. A possible explanation for this discrepancy might be the pathological behavior of $\calL$ on $L^2$ causes numerical ambiguity between the essential spectrum and eigenvalue.

\mbox{}

\noindent \textit{3. Stability and instability in fluid dynamics.}

\vspace{0.2cm}

The stability problem is a central topic in fluid dynamics and is notoriously difficult. The linear stability of steady states is well understood only for unidirectional flows, such as shear flows and radial vortices. There is a vast literature in this direction, and we refer the reader to \cite{arnold1965,Bedrossian2019,friedlander1997nonlinearinstability,Grenier2016SpectralInstability,Grenier2020,Rayleigh_Stability_Instability1879,Thomson_1880,tollmien1935allgemeines} for representative works, without attempting to be exhaustive.

However, much less is known beyond unidirectional flows. A few exceptions include, but are not limited to, certain cat's-eye flows \cite{catseyeYudovich2000}, studied via the method of averaging; the Kelvin-Stuart vortices \cite{liao2023stability} and the Euler-Poisson system \cite{zbMATH07656138}, analyzed using the index theory \cite{LinZeng22} and the spectral theory for separable Hamiltonian systems \cite{zbMATH07656138}. Precisely, in \cite{LinZeng22}, the authors developed a fairly general theory for counting unstable eigenvalues of linearized operators written in Hamiltonian form. In many cases, however, this theory is difficult to apply in practice because the relevant indices are hard to compute rigorously. In addition, \cite{zbMATH07656138} provides a precise counting formula for unstable modes of separable Hamiltonian systems; this result is also not completely general, since it requires the nonnegativity of a certain block. Interested readers can see also \cite{Lin03_planeflowSIAM,Lin04_planeflowIMRN,Lin2004_planeflow_CMP} for more related instability criteria and their applications in the ideal plane flows.

Very recently, in \cite{cao2026instability}, the authors introduced an alternative criterion for proving the spectral instability of the Taylor-Green vortex in two-dimensional ideal fluids, together with the necessary computer-assisted analysis. Compared with \cite{zbMATH07656138}, the theory developed in \cite{cao2026instability} does not require the nonnegativity of a certain block, and therefore has the potential to apply to more general situations.

Nevertheless, we emphasize here that all of the spectral theories established previously heavily rely on the Hamiltonian structure, and to the best of the authors' knowledge, it is largely open when the linear system does not enjoy the Hamiltonian structure (e.g., patch solutions).

\mbox{}

\noindent \textit{4. Challenges and novelties of the proof.}

\vspace{0.2cm}

The proof of Theorem \ref{main thm: linear} is based primarily on the framework of index theory for Hamiltonian linear systems developed by Lin and Zeng \cite{LinZeng22}, see also Theorem \ref{thm: index thm}. The key difficulties lie in adapting this framework to the "nearly" Hamiltonian linear system \eqref{linearized operator: almost Hamiltonian form} and in computing the corresponding indices in the identity \eqref{index theorem identity}.

\vspace{0.2cm}

\noindent (1) From "nearly" Hamiltonian to Hamiltonian.

\vspace{0.1cm}

In the literature, it is well known that if a steady state for 2D Euler equation satisfies a pointwise algebraic relation between the vorticity and the stream function of the form $\omega^* = F(\psi^*)$, with $F$ surjective, then the corresponding linearized equation typically admits a Hamiltonian formulation (see \cite{cao2026instability, liao2023stability} for examples). In contrast, this algebraic relation holds only within the unit disk for Lamb-Chaplygin dipole (see \eqref{lamb dipole: algebraic relation}). As a result, it can only help us to write the linearized operator $\calL$ \eqref{linearized op: def} in a "nearly" Hamiltonian structure \eqref{linearized operator: almost Hamiltonian form}.

To fit \eqref{linearized operator: almost Hamiltonian form} in the genuine Hamiltonian framework, 
we observe the \textit{a priori localization} of unstable eigenfunction: any $L^2(\RR^2)$ unstable eigenfunction of $\calL$ must be compactly supported in the unit disk $B$, using the skew-adjointness of $J$ (see \eqref{linearized operator: almost Hamiltonian form} for its definition). This reduces the mode stability of $\calL$ to the restricted operator $\calL \circ \mathbf{1}_B$, for which the index theory of Lin and Zeng \cite{LinZeng22} is applicable.

\mbox{}

\noindent (2) Calculating the indices: offset by the symmetries.

\vspace{0.1cm}

Although the problem is reduced to a linear Hamiltonian formulation, the relevant indices are generally difficult to compute. In particular, $\tilde L=L\circ \mathbf{1}_B$ possesses three negative directions $ n^-(\tilde L)=3$, which seems to be large enough to present a tough obstacle to establishing the mode/spectral stability. These negative directions, however, can be offset by the symmetries of 2D Euler equation and hence do not contribute to any unstable eigenvalue. See Section \ref{sec 24} for details.

\mbox{}

\subsubsection{Orbital stability of Lamb-Chaplygin dipole with quantitative description under sign and symmetry condition}

\mbox{}

\vspace{0.1cm}

Our second result is a quantitative improvement of the pioneering orbital stability results from Abe-Choi-Jeong \cite{abe2022stability, abe2025stability}. The initial data we consider lives in
\be 
\calX_{\text{odd},+} = \left\{ \omega \in (L^1_{x_2} \cap L^2)(\RR^2): \omega(x_1, -x_2) = -\omega(x_1, x_2),\,\omega\big|_{\RR^2_+} \ge 0\right \},
\label{Xodd+}
\ee 
which is the same as \cite{abe2025stability}, namely with odd symmetry, non-negativity, and possibly infinite mass.
In particular, $E[\omega] < \infty$ for $\omega \in \calX_{\text{odd}, +}$ from \eqref{energy ineq}. 
Since there is no known well-posedness theory of strong solutions for \eqref{eq: Euler} in $\calX_{\text{odd},+}$, we define the following class of admissible solutions to which our theorem applies.

\begin{definition}[Admissible solution] \label{def: admissible solu} We call $\omega = \omega(t, x)$ an admissible solution of \eqref{eq: Euler} with initial data $\omega_0 \in \calX_{\text{odd},+}$ if 
\begin{enumerate}
 \item $\omega \in C^0([0,\infty),  L^2(\RR^2) \cap L^1_{x_2}(\RR^2))$, $\omega\big|_{t= 0} = \omega_0$, and $\omega(t,\cdot) \in \calX_{\text{odd},+}$ for all $t \ge 0$; 
  \item $\omega$ is a global-in-time weak vorticity solution of \eqref{eq: Euler}, namely for any $\varphi \in C^\infty_c([0, \infty) \times \RR^2)$, we have 
 \be \int_0^\infty \int_{\RR^2} \omega \left(\pa_t + \nabla^\perp \Delta^{-1} \omega \cdot \nabla\right)\varphi dxdt + \int_{\RR^2} \omega_0 \varphi(0, \cdot) dx = 0.
 \label{weak sol: def}
 \ee
 \item $E[\omega(t)] = E[\omega_0], \;K[\omega(t)] = K[\omega_0], \;I[\omega(t)] = I[\omega_0]$ for all $t \ge 0$. 
\end{enumerate}
\end{definition}

\begin{remark}[Existence of admissible solution] Classical weak solution theory implies that for any $\omega_0 \in \calX_{\text{odd},+}$, there exists at least one admissible solution in the sense of Definition \ref{def: admissible solu}, c.f. \cite[Proposition B.1]{abe2025stability}.
\end{remark}

Next, we define the following nonlinear distance functional on $L^1_{x_2} \cap L^2(\RR^2)$ as
\be 
 \bmrmd[ \omega] := c_L^{-2} \| \errl \cdot \omega\|_{L^1(\RR^2)} + \| \omega \|_{L^2(\RR^2)}^2.\label{eqdefdist} 
\ee
This distance functional is comparable with $(L^1_{x_2} \cap L^2)(\RR^2)$ norm in a nonlinear fashion (see Lemma \ref{lem: dist}). Now we can state our main theorem. 

\begin{theorem}[Quantitative orbital stability of Lamb-Chaplygin dipole with sign and symmetry conditions]
\label{thm: orbital stability sec3}
There exist $\delta_1 > 0$ and $C_{\rm{stab}}, C_1, C_2> 0$, such that the following holds: For any $\omega_0 \in \calX_{\mathrm{odd},+}$ satisfying
\be 
  \bmrmd[ \omega_0 - \wl] \le \delta_1, \label{initial smallness}
\ee 
then any admissible solution $\omega(t, \cdot)$ of \eqref{eq: Euler} with initial data $\omega_0$ in the sense of Definition \ref{def: admissible solu} satisfies 
\be
  \bmrmd[\omega(t, \cdot) -  (1+\a(t)) \wl (\cdot - \beta(t)) ] \le C_{\rm{stab}} \bmrmd[ \omega_0 - \wl],\quad \text{for all } \,\,t \ge 0,  \label{orbital stability: main result}
\ee
for some $(\a, \beta) \in C^1_{loc}(\RR\mapsto \RR \times \RR)$ with the bounds
\begin{align} 
  |\a(t)|&\le C_1 \bmrmd[ \omega_0 - \wl]^\frac 12, \quad \text{ for all } \,\,t \ge 0. 
  \label{smallness: parameters alpha}\\
 |\beta'(t)-1| + |\a'(t)| &\le C_2 \bmrmd[ \omega_0 - \wl]^\frac 12, \quad \text{ for all } \,\,t \ge 0. \label{smallness: parameters beta}
\end{align}
\end{theorem}

\mbox{}

\noindent \textit{Comments on Theorem \ref{thm: orbital stability sec3}}.

\mbox{}

\noindent \textit{1. Quantitative orbital stability.}

\vspace{0.2cm}

Our result improves \cite{abe2022stability,abe2025stability} quantitatively in the following two senses. Firstly, in terms of the estimate of orbital stability, \cite{abe2022stability,abe2025stability} proves
\be 
 \inf_{\beta \in \RR} \| \omega(t,\cdot) - \wl(\cdot - \beta \mathbf{e}_1) \|_{(L^1_{x_2} \cap L^2)(\RR^2)} = o_{ \| \omega_0 - \wl \|_{(L^1_{x_2} \cap L^2)(\RR^2)} \to 0}(1).  \label{orbital stability: main result AC}
\ee
Our stability estimate \eqref{orbital stability: main result} refines the $o(1)$ to be \textit{the linear dependence} on the size of initial perturbation. We stress this linear dependence is sharp in the sense that one cannot expect $o\left(\bmrmd[\omega_0 - \wl]\right)$ to be the right hand side of \eqref{orbital stability: main result} with whatever $\bmrmd$. This sharpness indicates the necessity of modulating the amplitude $\a(t)$ (in view of its $\bmrmd^\frac 12$ bound \eqref{smallness: parameters alpha}) and our special choice of $\bmrmd$ \eqref{eqdefdist}. 
Secondly, we can determine the modulation parameters with $C^1$-smoothness in time, and uniformly bound the travelling speed as \eqref{smallness: parameters beta}, which also seems to be optimal for orbital stability control. 

We expect this quantitative information to assist further dynamical analysis near the Lamb-Chaplygin dipole, for example, in the construction of small-scale created solutions as in \cite{zbMATH07488949,jeong2025superlinear}. 

\vspace{0.2cm}

\mbox{}

\noindent \textit{2. Variational approach vs. Quantitative Lyapunov method.}

\vspace{0.2cm}

In \cite{abe2022stability,abe2025stability}, the orbital stability was proven via the variational approach. The authors identified Lamb-Chaplygin dipole as the unique (up to $x_1$-translation) global optimizer of a constrained variational problem, and verified the compactness of minimizing sequence to derive the qualitative bound. This robust strategy also finds application in other orbital stability problems for Euler equation, see \cite{zbMATH07782026,zbMATH07599255,zbMATH07941365,Wang_ConcentratedVortexPair24,Wang_StabilityOnDisk25}. Nevertheless, this method seems to be hard to provide a quantitative bound, and the center $\beta(t)$ is determined by minimization \eqref{orbital stability: main result AC} which can be noncontinuous and hard to evaluate.

In this work, we exploit the quantitative Lyapunov method from \cite{zbMATH03955503} where Weinstein proved orbital stability of ground states for several nonlinear dispersive equations. The main idea is to design a Lyapunov functional controlling the perturbation quantitatively. In particular, proving its coercivity is usually the key, often requiring spectral analysis of specific self-adjoint operators. 
For 2D incompressible Euler model, this strategy was also employed in \cite{liao2023stability} to study the orbital stability of Kelvin-Stuart cat's eye flow, which is a steady state without any known variational characterization. 

\vspace{0.2cm}

\mbox{}

\noindent \textit{3. New ingredients in the proof.}

\vspace{0.1cm}

\noindent (1) Nonvanishing but coercive first-order variation of Lagrangian functional.

\vspace{0.1cm}

The natural candidate of Lyapunov functional is $\calF[\omega] - \calF[\wl] =: \calQ[\omega - \wl]$ with the \textit{Lagrangian functional} $\calF[\omega]$ \eqref{functional: def} of which $\wl$ is a critical point. However, due to the compact support of $\wl$, the first-order variation $\frac{\delta\calF}{\delta \omega}\big|_{\omega=\wl}$ only vanishes inside $\supp \wl = B$, so that $\calQ$ contains both linear and quadratic parts. 

Luckily, $\frac{\delta\calF}{\delta \omega}\big|_{\omega=\wl}$ has a favorable sign \eqref{first variation}, which implies linear coercivity of $\calQ[h]$ in the exterior region assuming the non-negativity of $h$ on the upper half-plane. The coercivity of the interior quadratic form $ \calQ[h] = (\tilde \calS h, h)$ \eqref{linear op: cutoff version} can be derived via spectral analysis. Combining them with a careful spatially decoupled analysis, we can obtain the coercivity of $\calQ$ (see Proposition \ref{prop: coercivity}). We note that this determines the choice of distance functional $\bmrmd$. 

Besides, motivated by Weinstein \cite{zbMATH03955503}, we show that the coercivity retains after replacing the unstable eigenfunction direction by $\mathbf{1}_B x_2$ (Proposition \ref{prop: calS coercivity}), thanks to the algebraic relations $\tilde{\calS} \wl=-\frac12\mathbf{1}_B x_2$ and $(\mathbf{1}_B x_2,\wl)_{L^2(\RR^2_+)}>0$.

\vspace{0.2cm}

\noindent (2) Quadratic impulse correction in Lyapunov functional.

\vspace{0.1cm}

The functional $\calQ$ \eqref{def: Q} above is not yet a Lyapunov functional after injecting the modulation analysis, due to the quadratic terms related to the amplitude modulation parameter $\alpha(t)$ \eqref{eqdeficiency of Q}. We observe that the \textit{square of the impulse difference} produces an exact cancellation of the related terms, so that the Lyapunov functional $\calA(t)$ \eqref{eqdefcalAt} is coercive.

\vspace{0.2cm}

\noindent (3) Modulation analysis with low-regularity orthogonal directions.

\vspace{0.1cm}

The orthogonal directions $\mathbf{1}_B x_2, \partial_{x_1}\wl$ for coercivity of $\calQ$ (Proposition \ref{prop: calS coercivity}) do not belong to $H^1(\RR^2)$. Such low regularity is insufficient to introduce modulation parameters by the standard implicit function theorem and to control the evolution. To address the difficulty, following the idea of \cite[Lemma 2.11]{zbMATH08065795}, we impose modified orthogonality conditions by replacing them with smooth approximations  $W_1,W_2 \in C^\infty_c(\RR^2) \subset H^\infty(\RR^2)$. As a small $L^2$-perturbation, these new orthogonal directions still guarantee the coercivity and the non-degeneracy in the modulation setup. 

We mention other remedies for such low-regularity issue in modulational analysis. In Weinstein's work \cite{zbMATH03955503}, modulation is introduced through minimization, which works for rough orthogonal directions but cannot ensure uniqueness and $C^1$-continuity of the parameters. In \cite{jeong2025superlinear}, Jeong, Yao, and the third author introduced a $C^1$ approximate center by choosing a regular orthogonal direction based on filamentation of the perturbed Lamb-Chaplygin dipole.
This also implies a uniform control of the center moving velocity $\beta'$ \cite[(3.13)]{jeong2025superlinear}, while our analysis seems to imply a better bound \eqref{smallness: parameters beta}.

\vspace{0.2cm}

\mbox{}

\noindent \textit{4. Infinite-mass, odd-symmetry and non-negativity conditions.}

\vspace{0.2cm}

As in \cite{abe2025stability}, we do not require $L^1(\RR^2)$-integrability of initial data, thanks to the fact that the energy interpolation inequality \eqref{energy ineq} does not involve $L^1(\RR^2)$, and that the modulation analysis is fully in $L^2(\RR^2)$. 

In contrast, the non-negativity in the upper half-plane is used in an essential way to ensure the exterior linear coercivity of $\calQ$ \eqref{first variation}, which is a cornerstone of the coercivity of Lyapunov functional. 

Furthermore, if we drop the odd-symmetry condition, not only do we lose the energy interpolation inequality \eqref{energy ineq} and cannot propagate non-negativity, but also the dynamics could be very different by admitting rotation.

\subsection{Structure of the paper.}

\mbox{}

\vspace{0.2cm}

\noindent \textit{1. Sketch of the proof of Theorem \ref{main thm: linear}.}

\vspace{0.1cm}

Section \ref{sec 2} is devoted to the proof of Theorem \ref{main thm: linear} by the index theory for linear Hamiltonian form, and we study the spectral property of $\calL$ in $L^2(\RR^2)$ sense from Section \ref{sec 21} to Section \ref{sec 24}, then turn to the property of $\calL_\a$ in $X_\a$ in Section \ref{sec 25}. Precisely, in Section \ref{sec 21}, we use the algebraic relation \eqref{lamb dipole: algebraic relation} to rewrite \eqref{linearized op: def} in a more structured form \eqref{linearized operator: almost Hamiltonian form}. In Section \ref{sec 22}, we prove that every unstable eigenfunction of $\calL$ is compactly supported in the unit disk $B$ (Corollary \ref{lemma: reduce eigenproblem to B}) and thus reduces the problem to the operator $\tilde{\calL}:=\calL\circ \mathbf{1}_B := J \tilde L$ enjoying the Hamiltonian structure. After reviewing the index theory in Section \ref{sec 23}, we turn in Section \ref{sec 24} to the operator $\tilde \calL$, where we prove that $n^-(\tilde L)=3$ by an explicit analysis of $\tilde L= L \circ \mathbf{1}_B$ (Lemma \ref{lemma: n-(L)}), and show that $k_{0}^{\le 0}(\tilde{\calL})\ge 3$ using the symmetries of the 2D Euler equation (Lemma \ref{lemma: non-positive directions of the bilinear form of L}), so Theorem \ref{main thm: linear} (i) simply follows from these facts together with the index theory. Finally, in Section \ref{sec 25},  we use the semigroup theory generated by incompressible flow (Lemma \ref{lemma: spectrum of the transport op on weighted L2 is on the imaginary axis}), compactness theorem (Lemma \ref{Lemma: spectrum of compactly-perturbed operator}) and the mode stability of $\calL$ in $L^2(\RR^2)$ (Theorem \eqref{main thm: linear} (i)) to conclude the spectral stability of $\calL_\a$ in $X_\a$, i.e. Theorem \ref{main thm: linear} (ii).

\mbox{}

\vspace{0.1cm}

\noindent \textit{2. Sketch of the proof of Theorem \ref{thm: orbital stability sec3}.}

\vspace{0.1cm}

Section \ref{sec3} is devoted to the proof of Theorem \ref{thm: orbital stability sec3} by means of the quantitative Lyapunov method. The essential part of the proof is carried out in Section \ref{sec 33}, where we verify the coercivity of $\calQ$, defined as the sum of the first and second variations of the Lagrangian functional $\calF$ in \eqref{functional: def}. Exploiting the nonnegativity of the first-order variation together with the exterior error of the Lamb--Chaplygin dipole introduced in \eqref{error: def}, we first reduce the problem to the spectral study of $\tilde \calS$ in \eqref{linear op: cutoff version} on $L^2(\RR^2)$. The coercivity of $\tilde \calS$ is then obtained through a suitable choice of orthogonality conditions and the method of Lagrange multipliers, which has been discussed in Section \ref{sec 32}. In Section \ref{sec 34}, we construct the Lyapunov functional $\calA(t)$ \eqref{eqdefcalAt}, a combination of the Lagrangian difference $\calF[\omega]-\calF[\wl]$ and quadratic impulse correction $\left(I[\omega]-I[\wl]\right)^2$, and then complete the proof of Theorem \ref{thm: orbital stability sec3}  using the coercivity of $\calQ$ together with the modulation argument. 


\subsection{Notations}
For any $\Omega \subset \RR^2$ and $w \ge 0$ on $\Omega$, we denote by $L_w^2(\Omega)$ the space of all functions $f$ with $\| f\|_{L_w^2(\Omega)}^2 := \int_{\Omega} |f(\mathbf x)|^2 w(\mathbf x) d \mathbf x < \infty$, and by $L_{x_2}^1(\Omega)$ the space of all functions $ f$ with $\| f\|_{L_{x_2}^1(\Omega)} := \int_{\Omega} |f(\mathbf x)| |x_2| d \mathbf x < \infty$. Additionally, we denote $B_r:= B(0,r)$, and $B_r^c := \RR^2 \setminus B_r$. Moreover, if $A: \calX \to \calY$ is a linear operator, then we denote $\rho(A)$ as the resolvent set of $A$, $\sigma(A)$ as the spectrum of $A$, $\sigma_{\text{c}}(A)$ as the continuous spectrum of $A$, $\sigma_{\text{r}}(A)$ as the residual spectrum of $A$, $\sigma_{\text{p}}(A)$ as the point spectrum of $A$,  $\sigma_{\text{ess}}(A)$ as the essential spectrum of $A$, and $\sigma_{\text{disc}} (A)$ as the discrete spectrum of $A$. We denote constants generically by $C>0$, the values of which may vary from line to line.


\mbox{}

\begin{acknowledgement}
    The work of Z.L. is part of the ERC starting grant project FloWAS that has received funding from the European Research Council (ERC) under the Horizon Europe research and innovation program (Grant agreement No. 101117820). P.S. is supported by NSF Grants DMS-2205590 and DMS-2512878, the Choi Family Gift fund and Dr. Mike Yan Gift fund. T.Z. is supported by MOE Tier 1 Grant A-8004147-00-00.
 The authors gratefully thank In-Jee Jeong and Yao Yao for suggesting this problem and valuable discussions during the course of this research. Special thanks are also extended to  Daniel Boutros, Xiaoyutao Luo and Guolin Qin for their helpful discussions and generous encouragement.
\end{acknowledgement}
\section{On the spectral stability of the Lamb-Chaplygin dipole}\label{sec 2}
This section is dedicated to the proof of Theorem \ref{main thm: linear}. In Sections \ref{sec 21} - \ref{sec 24}, we establish the mode stability of $\calL$ (i.e., no eigenvalues with strictly negative real parts) with respect to the space $L^2(\RR^2)$. In Section \ref{sec 25}, we establish the spectral stability of $\calL$ in the weighted $L^2$-spaces $X_\a$. 
\subsection{"Nearly" Hamiltonian form of linearized operator \eqref{linearized op: def}.}
\label{sec 21}
As for the linearized operator $\calL$ defined \eqref{linearized op: def}, we are going to make full use of the structure of Lamb-Chaplygin dipole to rewrite it in a "nearly" Hamiltonian form. Precisely, recalling \eqref{lamb dipole: algebraic relation} and \eqref{moving frame: stream function}, together with the fact that $\nabla^\perp f\cdot\nabla g = -\nabla^\perp g\cdot\nabla f$, we obtain that 
\begin{align}
    \calL h 
    & = - \mathbf{e}_1 \cdot \na h + \na h \cdot \na^\perp \Delta^{-1} \wl + \na \wl \cdot \na ^\perp \Delta^{-1} h \notag\\
    & = \na h \cdot \na^\perp ( x_2 + \Delta^{-1} \wl) + \na \wl \cdot \na ^\perp \Delta^{-1} h \notag\\
    & = \na h \cdot \na^\perp \psilm  - c_L^2 \na \left(\psilm \mathbf{1}_B \right) \cdot \na^\perp \Delta^{-1} h \notag  \\
    & = \na \psilm \cdot \left( - \na^\perp h  - c_L^2 \mathbf{1}_B \na^\perp \Delta^{-1} h \right)  - c_L^2 \psilm \na \mathbf{1}_B \cdot \na ^\perp \Delta^{-1} h \notag \\
    & = \na \psilm \cdot \na^\perp \left( -h - c_L^2 \mathbf{1}_B \Delta^{-1} h  \right)
    + c_L^2 \left( \na \psilm \cdot \na^\perp \mathbf{1}_B \right) \Delta^{-1} h 
    \label{halmitonian 1}
    \\
    & = - \na^\perp \psilm \cdot \na \left( h + c_L^2 \mathbf{1}_B \Delta^{-1} h  \right),
    \label{halmitonian 2}
\end{align}
where \eqref{halmitonian 1} and \eqref{halmitonian 2} follow from the fact that $\psilm \Big|_{\partial B} \equiv 0$ and $\na^\perp \psilm \perp \mathbf{n}$ on $\partial B$ ($\mathbf n$ is the normal vector on $\partial B$). Hence it means that $\calL$ \eqref{linearized op: def} in fact takes a "nearly" Hamiltonian form in the sense that
\begin{align}\label{linearized operator: almost Hamiltonian form}
\begin{split}
    &\L h = \nabla^\perp\psilm\cdot\nabla \left(h - c^2_L\one_B \left( -\Delta \right)^{-1} \right) = JLu,\\
    \text{ with } \quad  &J :=\nabla^\perp\psilm\cdot\nabla,\quad L:= I - c^2_L\one_B \left(-\Delta \right)^{-1} .
\end{split}
\end{align}

\begin{remark}
    Since the velocity field $\nabla^\perp\psilm\in W^{1,\infty}(\RR^2)$, in $L^2(\RR^2)$ sense, it is standard that $J$ is skew-adjoint (i.e., $J^* = -J$) as the realization of its expression equipped with its maximal domain
    \[
       D(J):= \{f\in L^2(\RR^2): \nabla^\perp\psilm\cdot\nabla f\in L^2(\RR^2)\text{ in the distributional sense}\}.
    \]
    However, the operator $L$ is not self-adjoint (i.e., $L^* = L$), even not symmetric in $L^2(\RR^2)$, which prevents us from using the index theory established in \cite{LinZeng22} to study the related spectrum, and this is the reason why we call it of "nearly" Hamiltonian form.
\end{remark}

\subsection{Reduction to Hamiltonian form}
\label{sec 22}
Our first goal is to study the unstable eigenmodes of the linearized operator $\calL$ with form \eqref{linearized operator: almost Hamiltonian form} in $L^2(\RR^2)$ and prove Theorem \ref{main thm: linear} (i), which will be coverd from Section \ref{sec 22} to Section \ref{sec 24}. To handle this problem, our first target is to reduce the operator \eqref{linearized operator: almost Hamiltonian form} into a Hamiltonian one, a step that relies crucially on the structure of the velocity field. More precisely, we establish the following property of $J$ stated in a general setting.

\begin{lemma}[Localization of unstable eigenmodes]
\label{a general eigenproblem lemma}
    Let $\Omega = B^c$, and let the incompressible vector field $\mathbf v\in C^\infty(\bar\Omega;\RR^2)$ satisfy $\nabla\cdot \mathbf v = 0$ in $\Omega$, $
\mathbf v\cdot e_r=0$ on $\pa\Omega$, and the linear growth control
    \begin{equation}\label{prop: general eigen_far field growth control of v}
        \sup_{R>2}\frac{1}{R}\|\mathbf v\cdot \mathbf e_r\|_{L^\infty(R<|x|<2R)}<+\infty. 
    \end{equation}
    If $u\in L^2(\Omega)$ solves
    \[
      \mathbf  v\cdot \nabla u = \lambda u,\quad \text{in }\calD'(\Omega).
    \]
    Then,
    \[ 
        \mathrm{Re}(\lambda)\|u\|^2_{L^2(\Omega)} = 0.
    \]
    In particular, if $\mathrm{Re}(\lambda)\neq 0$, then $u=0$ almost everywhere in $\Omega$.
\end{lemma}
\begin{proof}
Since $\mathbf v$ is smooth and $u\in L^2(\Omega)$ solves     
    \[
      \mathbf  v\cdot \nabla u = \lambda u,\quad \text{in }\calD'(\Omega),
    \]
the standard renormalization property for transport equations gives
\[
    \mathbf v\cdot\nabla|u|^2 = 2\mathrm{Re}(\lambda)|u|^2, \quad \text{in }\calD'(\Omega).
\]
Equivalently, for any $\eta\in C^\infty_0(\Omega)$, it holds that
\[
    2\mathrm{Re}(\lambda)\int_{\Omega}|u|^2\eta \;d\mathbf x= -\int_\Omega (\mathbf v\cdot\nabla\eta )|u|^2\;d\mathbf x.
\]
Now we choose two smooth radial cutoff functions $a_\delta(r)$ and $b_R(r)$ taking values in $[0,1]$, such that
\[
    a_\delta(r)=0\;\text{ for } r<1+\delta,\quad   a_\delta(r)=1\;\text{ for } r>1+2\delta,\quad |a_\delta'(r)|\leq\frac{C}{\delta},
\]
and
\[
    b_R(r)=1\;\text{ for } r<R,\quad   b_R(r)=0\;\text{ for } r>2R,\quad |b_R'(r)|\leq\frac{C}{R}.
\]
Then, we define the test function 
\[
    \eta_{\delta,R}(\mathbf x):= a_\delta(|\mathbf x|)b_R(|\mathbf x|)\in C^\infty_0(\Omega).  
\]
Applying the identity above with $\eta = \eta_{\delta,R}$, we obtain
\begin{equation}\label{prop: general eigen_test equation}
    2\mathrm{Re}(\lambda)\int_\Omega |u|^2\eta_{\delta,R}\;dx = -\int_\Omega (\mathbf v\cdot\nabla\eta_{\delta,R} )|u|^2\;d\mathbf x.
\end{equation}
Since $\eta_{\delta,R}$ is radial, we compute
\[
   \mathbf v\cdot\nabla\eta_{\delta,R} = (\mathbf v\cdot \mathbf e_r)(a'_\delta b_R+a_\delta b'_R),
\]
which is supported on the annuli $A_\delta:=\{\mathbf x:1\leq |\mathbf x|\leq 1+\delta\}$ and $A_R:=\{\mathbf x: R\leq|\mathbf x|\leq 2R\}$. On $A_\delta$, since $\mathbf v\cdot \mathbf e_r\big|_{r=1}=0$, by smoothness of $\mathbf v$ we have
\[
    |\mathbf v\cdot \mathbf e_r|\leq C\delta.
\]
Then, by the estimate $|a'_\delta|\leq C\delta^{-1}$, we have
\[
    \da{\int_{A_\delta}(\mathbf v\cdot\nabla\eta_{\delta,R})|u|^2\;dx} = \da{\int_{A_\delta}(\mathbf v\cdot \mathbf e_r)(a'_\delta b_R)|u|^2\;dx}\leq C\|u\|^2_{L^2(A_\delta)}\to0,\quad \delta\to 0^+.
\]
Similarly, on $A_R$, due to the control \eqref{prop: general eigen_far field growth control of v} and $|b'_R|\leq CR^{-1}$, we have
\[
\da{\int_{A_R}(\mathbf v\cdot\nabla\eta_{\delta,R})|u|^2\;dx} = \da{\int_{A_\delta}(\mathbf v\cdot \mathbf e_r)(a_\delta b'_R)|u|^2\;dx}\leq \frac{1}{R}\|u\|^2_{L^2(A_R)}\to0,\; R\to +\infty.
\]
Finally, by taking $\delta\to 0^+$ and $R\to +\infty$ in \eqref{prop: general eigen_test equation}, we obtain
\[
   2\mathrm{Re}(\lambda)\int_\Omega|u|^2\;dx = 0.
\]
Consequently, we have concluded the proof.
\end{proof}

Based on Lemma \ref{a general eigenproblem lemma}, as a crucial corollary, we can verify that all of the $L^2(\RR^2)$ unstable eigenmodes of $\calL$ should be compactly supported in the unit disk $B$:

\begin{corollary}[Reduction to the unit disk]
\label{lemma: reduce eigenproblem to B}
Let $\varphi\in L^2(\RR^2)$ be an eigenfunction of $\calL$ with respect to the eigenvalue $\lambda$ with $\mathrm{Re}\,(\lambda)\neq 0$. Then, it holds that $\varphi\mid_{B^c} =0$ almost everywhere, i.e., $\mathrm{supp}(\varphi)\subset B$. 
\end{corollary}
\begin{proof}
Assume that $\calL \varphi = \lambda \varphi$ with $\varphi \in L^2(\RR^2)$ with $\mathrm{Re}\,(\lambda)\neq 0$. If we decompose $\varphi$ into $\varphi = \varphi_{in} + \varphi_{out}$ with $\varphi_{in} = \varphi \mathbf{1}_B$ and $\varphi_{out} = \varphi \mathbf{1}_{B^c}$, then $J \varphi_{out} = \lambda \varphi_{out}$ with $\varphi_{out} \in L^2(\RR^2)$ and $\text{supp } (\varphi_{out}) \subset B^c$. Note that
\begin{equation}
    \mathbf e_r\cdot\nabla^\perp\psilm = -\dr{1-\frac{1}{r^2}}\cos(\theta),\quad r\geq 1.
\end{equation}
Hence we can apply Lemma \ref{a general eigenproblem lemma} with $u:= \varphi|_{B^c}$ and $\mathbf v:=\nabla^\perp\psilm$, together with the assumption that $\mathrm{Re}\,(\lambda)\neq 0$, to conclude that $\varphi_{out} =0$, and thus we have finished the proof.
\end{proof}

 \begin{remark}[Simplify the operator \eqref{linearized operator: almost Hamiltonian form} into the Hamiltonian form]
 \label{rmk: simplify to Hamiltonian form} As a result of Corollary \ref{lemma: reduce eigenproblem to B}, to study the unstable eigenvalue of $\calL$ on $L^2(\RR^2)$, it is equivalent to the study of the operator
 \be
 \tilde \calL := \calL \circ \mathbf{1}_B = J \tilde L, \; \text{ with } \; \tilde L = L \circ \mathbf{1}_B = \mathbf{1}_B - c_L^2 \mathbf{1}_B \left( -\Delta \right)^{-1} \mathbf{1}_B.
 \label{tilde calL}
 \ee
Note that $\tilde L : L^2(\RR^2) \to L^2(\RR^2)$ is a bounded self-adjoint operator, $\tilde \calL$ is of Hamiltonian form, and then we can then turn to the index theory established in \cite{LinZeng22} to study the spectral property of $\tilde \calL$.
\end{remark}

\vspace{0.1cm}

\subsection{Review of the index theory \cite{LinZeng22}}
\label{sec 23}
Before we continue the study of the spectrum of $\tilde\calL$, let us recall the index theory for the linear Hamiltonian system 
\begin{equation}\label{Linear Hamiltonian system}
\pa_t u = JL u,\quad u\in X,
\end{equation}
established by Lin and Zeng \cite{LinZeng22}. The theory relates the number of unstable modes of $JL$ to that of $L$, which are usually easier to study. Before we introduce their main result, some preliminaries are in order. In the system \eqref{Linear Hamiltonian system}, suppose that $X$ is a real Hilbert space with inner product $\scl{\cdot}{\cdot}$. Assume the following:
\begin{enumerate}[start=1,label={(\bfseries H\arabic*)}]
    \item\label{H1} $J: D(J) \subset X\to X$ is skew-adjoint, in the sense that $J^* = -J$.\\
    \item\label{H2} $L: X\to X$ is bounded and self-adjoint (i.e., $L^* = L$) such that $\scl{L\cdot}{\cdot}$ defines a bounded symmetric bilinear form on X. Moreover, there
    exists a decomposition of X into the direct sum of three closed subspaces
    \[
        X = X_-\oplus\ker(L)\oplus X_+,\quad n^-(L):=\dim(X_-)<+\infty,
    \]
    such that $\scl{Lu}{u}<0$ for all $u\in X_-\backslash\{0\}$ and $\scl{L u}{u}\geq \delta\|u\|^2$ for any $u\in X_+$ for some universal $\delta>0$.\\
    \item\label{H3} It holds that
    \[
        \{v\in X: \scl{v}{u}=0,\;\forall\;u\in X_-\oplus X_+\}\subset D(J).
    \]
\end{enumerate}

\mbox{}

We remark that the index $n^-(L)$ is well-defined (independent of the decomposition), which equals the negative dimensions of quadratic form $\scl{L\cdot}{\cdot}$. This standard result is an infinite-dimensional analogue of Sylvester's law of inertia.
For any eigenvalue $\lambda$ of $JL$, define the generalized eigenspace as
\begin{equation}
    E_\lambda:= \{u\in X: (JL - \lambda)^ku=0,\;\text{for some integer }k\geq 1\}.
\end{equation}
Then we let $k_r$ be the sum of algebraic multiplicities of positive eigenvalues of $JL$ and
$k_c$ be the sum of algebraic multiplicities of eigenvalues of $JL$ in the first quadrant, i.e.,
\[
    k_r:=\sum_{\lambda>0}\dim(E_\lambda),\quad k_c:=\sum_{\mathrm{Re}(\lambda)>0,\;\mathrm{Im}(\lambda)>0}\dim(E_\lambda).
\]
Given any subspace $S\subset X$, define $n^-(L|_S)$ and $n^{\leq 0}(L|_S)$ to be the the maximal negative and non-positive dimensions of $\scl{L\cdot}{\cdot}$ restricted to $S$, respectively. By the assumption on $L$, we have $n^-(L|_S)\leq n^{\leq 0}(L|_S)\leq n^-(L)<+\infty$ for any $S$ satisfying $S\cap \ker(L) = \{0\}$.
Then, for any purely imaginary eigenvalue $i\mu$ ($\mu>0$) of $JL$, define
\begin{equation}
    k_i^{\leq 0}:=\sum_{\mu>0}k^{\leq 0}(i\mu),\quad \text{with} \quad k^{\leq 0}(i\mu):=n^{\leq 0}(L|_{E_{i\mu}}).
\end{equation}
Finally, note that $\ker(L)\subset E_0$. Let $\tilde E_0$ be any subspace such that $E_0 = \ker(L)\oplus \tilde E_0$. We define
\[
   k_0^{\leq 0}:= n^{\leq 0}(L|_{\tilde E_0}).
\]
We remark that the definition of $k_0^{\leq 0}$ is independent of the choice of $\tilde E_0$. Now we are ready to introduce the index theorem:
\begin{theorem}[Index theory {\cite[Theorem 2.3]{LinZeng22}}]
\label{thm: index thm}
    Under the assumptions \ref{H1}, \ref{H2}, and \ref{H3}, the indices satisfy
    \begin{equation}\label{index theorem identity}
        k_r+2k_c+2k_i^{\leq0}+ k_0^{\leq0}=n^{-}(L).
    \end{equation}
\end{theorem}

\vspace{0.1cm}

\subsection{Mode stability of $\calL$ on $L^2(\RR^2)$}
\label{sec 24}

Based on the preliminaries established above, we are going to study the unstable eigenmodes of $\tilde \calL$ on $L^2(\RR^2)$ (which is equivalent to studying those of $\calL$ on $L^2(\RR^2)$), so that we can finally complete the proof of Theorem \ref{main thm: linear} (i). 

First, let us study the spectrum of $\tilde L$ defined in \eqref{tilde calL} and count $n^-(\tilde L)$. Note that the kernel of $\mathbf{1}_B (-\Delta)^{-1} \mathbf{1}_B$ satisfies
\begin{align}
    \frac{1}{4 \pi^2} \iint_{B \times B} \Big| \log |\mathbf x - \mathbf y| \Big|^2 d \mathbf x d \mathbf y <\infty, \label{Hilbert-Schmidt of tilde L}
\end{align}
and it follows that $\mathbf{1}_B \left( -\Delta \right)^{-1} \mathbf{1}_B: L^2(\RR^2) \to L^2(\RR^2)$ is a Hilbert-Schmidt operator \cite[Section 6.8.5]{helffer2013spectral} so that it is compact \cite[Theorem 9.4]{zbMATH00837307}. Together with the fact that $\sigma(\mathbf{1}_B) = \{ 0,1 \}$ and Weyl's theorem \cite[Theorem 10.11]{helffer2013spectral}, we obtain that
\[
\sigma_{\text{ess}} (\tilde L) \cap \{ \lambda \in \CC: \text{Re } \lambda <0  \} = \sigma_{\text{ess}} (\mathbf{1}_B) \cap \{ \lambda \in \CC: \text{Re } \lambda <0  \} = \emptyset,
\]
which yields
\[
\sigma\left( \tilde L \right) \cap \{ \lambda \in \CC: \text{Re } \lambda <0  \}  = \sigma_{disc} \left( \tilde L \right) \cap \{ \lambda \in \CC: \text{Re } \lambda <0  \}.
\]

Additionally, by the self-adjointness of $\tilde L: L^2(\RR^2) \to L^2(\RR^2)$, these eigenvalues are semisimple and real, and thus together with the spectral decomposition theorem of the self-adjoint operator \cite[Theorem 8.15]{helffer2013spectral}, 
\be
n^-(\tilde L ) = \sum_{\lambda <0}\dim \ker \left( \tilde L - \lambda \right).
\label{n-: expression}
\ee

Therefore, our next target is to calculate $n^-(\tilde L)$ by explicitly obtaining the precise unstable eigenmodes of $\tilde L$ on $L^2(\RR^2)$.

\begin{lemma}[Unstable eigenmodes of $\tilde L$ on $L^2(\RR^2)$]
\label{lemma: n-(L)}
As for the self-adjoint operator $\tilde L$ \eqref{tilde calL} on $L^2(\RR^2)$,
\[
\sigma(\tilde L) \cap \Big\{ \lambda \in \CC: \text{Re } \lambda <0 \Big\} = \left\lbrace 1- \frac{\mu_{1,1}^2}{\mu_{0,1}^2}  \right\rbrace,
\]
with the related unstable eigensubspace
\be
\mathrm{span } \Big\{ J_0 \left( \mu_{0,1} r \right),  J_1 \left( \mu_{0,1} r \right) \sin \theta, J_1 \left( \mu_{0,1} r \right) \cos \theta \Big\},
\label{eigenspace}
\ee
where $J_m$ is the $m$-th order Bessel function with first kind, and $\mu_{m,j}$ is the $j$-th positive zero of $J_m$. 
In particular, $n^-(\tilde L) =3$.
\end{lemma}

\begin{proof}
    Assume that $f \in L^2(\RR^2)$ satisfies $\tilde L f = \lambda f$ with $\lambda <0$, then with the definition of $\tilde L$ given in \eqref{tilde calL},
    \[
    \mathbf{1}_B f - c_L^2 \mathbf{1}_B (-\Delta)^{-1} \mathbf{1}_B f = \lambda f, \quad \Rightarrow \quad \text{supp } f \subset B, 
    \]
    so that $f$ solves the equation
    \be
    - \mathbf{1}_B (-\Delta)^{-1} f = \frac{\lambda-1}{c_L^2} f, \quad \text{ with }\; \text{supp } f \subset B.
    \label{eq1: f}
    \ee
    If we denote $\kappa^2: = \frac{c_L^2}{1-\lambda}$ and define $\psi:= -(-\Delta)^{-1} f$, then
    \be
    \begin{cases}
        \Delta \psi + \kappa^2 \psi =0, & \text{in } B,\\
        \Delta \psi = 0, & \text{in } B^c.
    \end{cases}
    \label{eigenfunction: eq 1}
    \ee
    Additionally, $\psi$ has the following pointwise estimate:
    \begin{align}
        |\psi(\mathbf x)| =  \frac{1}{2 \pi} \Bigg| \int_{\RR^2} \log |\mathbf x -\mathbf y | \mathbf{1}_B(\mathbf y) f(\mathbf y) d \mathbf y \Bigg| 
        \le C \| f\|_{L^2(\RR^2)} \Big( 1+ \log \la \mathbf x  \ra \Big), \text{ for any } \mathbf x \in \RR^2, \label{eigenfunction: boundedness}
    \end{align}
    with some uniform constant $C>0$, and thus it yields that
    \be
    \| f\|_{L^\infty} \le \frac{1}{\kappa^2} \big\|  \mathbf{1}_B(\mathbf x) \psi(\mathbf x) \big\|_{L^\infty} \le \frac{C}{\kappa^2} \| f \|_{L^\infty}.
    \label{f: pointwise estimate 2}
    \ee
    
    Next, we are going to solve \eqref{eigenfunction: eq 1}. By the spherical decomposition $\psi (\mathbf x) = \sum_{m \in \ZZ_{\ge 0}} R_m(r) e^{im \theta}$, together with the fact that $\Delta$ is invariant in each angular mode, we know that each $R_m$ solves
     \begin{equation}
        \begin{cases}
            R_m''(r)+\frac{1}{r}R_m'(r)+\dr{\kappa^2-\frac{m^2}{r^2}}R_m(r)=0, & \text{with } r<1,\\
             R_m''(r)+\frac{1}{r}R_m'(r)-\frac{m^2}{r^2}R_m(r) = 0, & \text{with } r>1.
        \end{cases}
        \label{eq: Rm}
    \end{equation}
    
    When $r<1$, the general solution is 
    \begin{equation}
        R_m(r) = c_{1,m}J_{m}(\kappa r)+c_{2,m} Y_m(\kappa r),
    \end{equation}
    where $J_m$ and $Y_m$ are the Bessel functions of the first kind and the second kind, respectively. By the pointwise boundedness of $f$ near the origin shown in \eqref{eigenfunction: boundedness} together with \eqref{eq1: f}, we know that $c_{2,m} = 0$ for each $m \ge 0$.
    
    When $r>1$, the general solution is 
    \begin{align}\label{lemma: Bessels_far field solutions}
    \begin{split}
        &R_m(r) = b_{1,m}\log(r)+b_{2,m},\quad\text{when }m=0,\\
        &R_m(r) = b_{1,m}r^{-m}+b_{2,m}r^{m},\quad\text{when }m\geq 1.
    \end{split}
    \end{align}
    By the requirement that the exterior part of $\psi$ is the actual Newtonian potential of a compactly supported source \eqref{eq1: f}, we claim that $b_{2,m}=0$ above for each mode $m \ge 0$. Indeed, when $m\geq1$, existence of algebraically growing part is not possible by the pointwise estimate given in \eqref{eigenfunction: boundedness}. When $m=0$, we observe that for any radial source $f$ which is supported inside $B$,
    \begin{align}
    \begin{split}
         (-\Delta)^{-1}f(r) &= -\frac{1}{2\pi}\int_{0}^{1}\int_{0}^{2\pi}\log|r-\rho e^{i\theta}|\rho f(\rho)\;d\theta d\rho \\
         &= -\dr{\int_0^1\rho f(\rho)\;d\rho}\log(r),\quad r>1,
    \end{split}
    \end{align}
    where we use the identity
    \[
       \frac{1}{2\pi} \int_0^{2\pi}\log|r-\rho e^{i\theta}|\;d\theta = \log(r),\quad r>1,
    \]
    by the mean value property of harmonic functions. Hence, there is no additive constant in the exterior part when $m=0$, i.e. $b_{2,0}=0$. 

    In addition, by \eqref{eigenfunction: boundedness} and \eqref{f: pointwise estimate 2}, it is easy to check that each $R_m(r)$ is $C^1$ continuous at $r=1$. In particular, $m\geq 1$, this requires
    \begin{align}
    \begin{split}
        &c_1 J_m(\kappa) = b_{1,m},\quad\text{and}\quad c_1\kappa J_m'(\kappa) = -mb_{1,m}, \\
        \Longrightarrow\;& \kappa J'_m(\kappa)+mJ_m(\kappa) = 0.
    \end{split}
    \end{align}
    Using the Bessel identity
    \be
        xJ'_m(x)+mJ_m(x) = xJ_{m-1}(x),
        \label{Bessel identity}
    \ee
    this becomes
    \[
        J_{m-1}(\kappa) = 0.
    \]
    Thus, it means that $\kappa = \kappa_{m,j} := \mu_{m-1,j}$, the $j$-th positive zero of $(m-1)$-th Bessel function $J_{m-1}$. 
    
    When it comes to $m=0$, since $\log(1)=0$, it requires \[
        J_0(\kappa) = 0\;\Longrightarrow\; \text{$\kappa$ can be chosen as }\kappa_{0,j} = \mu_{0,j}.
    \]
    Then, once we choose $b_1 = c_1\kappa J'_0(\kappa)$, the first derivative is matched at $r=1$. 
    
    In conclusion, the eigenvalue problem \eqref{eigenfunction: eq 1} is admissible only when $\kappa = \kappa_{m,j}= \mu_{m-1,j}$ with $m \ge 0$ (we denote $\mu_{-1,m}:= \mu_{0,m}$), and this requires the related eigenvalues to be (note that $c_L = \mu_{1,1}$)
    \[
    \lambda_{m,j} = 1- \frac{c_L^2}{\kappa^2} = 1- \frac{c_L^2}{\mu_{m-1,j}^2} = 1- \frac{\mu_{1,1}^2}{\mu_{m-1,j}^2}, \quad \text{ with } m \in \ZZ_{\ge 0} \text{ and } j \in \ZZ_{\ge 1}.
    \]
    Once we recall the distribution of zeros of the Bessel function of the first kind $J_m$, the possibility leading to $\lambda_{m,j} <0$ is only when $(m,j)= (0,1)$ and $(m,j)=(1,1)$ (both scenarios correspond to the eigenvalue $1- \frac{\mu_{1,1}^2}{\mu_{0,1}^2}$). In particular, the related eigen-subspace can be spanned as shown in \eqref{eigenspace}. Hence we have concluded the proof.
\end{proof}

Next we find that another essential ingredient is the explicit non-positive directions of $\scl{\tilde L\cdot}{\cdot}$ arising from the symmetries of the Lamb-Chaplygin dipole. Specifically, we note that 
\[
\wl,\quad \pa_{x_2}\wl, \quad \text{ and } \quad  \mathbf{x}^\perp\cdot\nabla\wl,
\]
are generated by the scaling of the magnitude, translation in the $x_2$ variable, and rotation, respectively, which give a lower bound $k_{0}^{\le 0}(\tilde \calL) \ge 3$. Precisely, we have the following results:
\begin{lemma}[Generalized kernel generated by symmetries]
\label{lemma: non-positive directions of the bilinear form of L}
    The following algebraic identity holds true
    \begin{align}
        \tilde L\wl = -c^2_L\one_B x_2,\quad \tilde L (\pa_{x_2}\wl) = -c^2_L\one_B,\quad \tilde L(\mathbf{x}^\perp \cdot \nabla\wl) = -c^2_Lx_1\one_B.
        \label{algebraic relation}
    \end{align}
    Specifically, we have
    \be
        \{\wl,\,\pa_{x_2}\wl,\, \mathbf{x}^\perp\cdot\nabla\wl\}\subset \Big\{f\in L^2(\RR^2):\tilde \calL^k u = 0,\;\exists\, k\in\mathbb{Z}_{\geq 1}\Big\} \setminus\ker(\tilde L),
    \label{generalized ker of calL}
    \ee
    and
    \be
    (\tilde{L} g, g) \le 0, \quad \mathrm{ for\,\, any\,\, } g \in \mathrm{span}\big\{ \wl,\,\pa_{x_2}\wl,\, \mathbf{x}^\perp\cdot\nabla\wl \big\}.
    \label{nonpositive of L: quadratic form}
    \ee
    In particular, $k_{0}^{\le 0}(\tilde \calL) \ge 3$. 
\end{lemma}

\begin{remark}
We do not take $\pa_{x_1} \wl$ into account. We note that $\pa_{x_1} \wl$ is generated by the translation invariance in $x_1$, which in addition satisfies $\tilde L (\pa_{x_1} \wl) =0$. Thus, by definition it is not a candidate for counting $k_{0}^{\le 0}(\tilde \calL)$.
\end{remark}

\begin{proof}[Proof of Lemma \ref{lemma: non-positive directions of the bilinear form of L}]
    The computation is straightforward, using explicit expressions of $\wl$ and $\psil$ given in \eqref{lamb dipole: algebraic relation} and \eqref{moving frame: stream function}.
    
    \noindent $\bullet$ As for $\wl$, we note that
    \[
       \tilde L\wl = \wl + c^2_L\one_B\Delta^{-1}\wl = \wl+c_L^2\one_B(\psilm-x_2) = -c^2_L\one_B x_2.
    \]
    Hence, $\wl\notin\ker(L)$ and 
    \begin{equation}
        \left( \tilde L \wl , \wl \right) = -c^2_L\left( w_L, \one_B x_2 \right)<0.
        \label{inner prod 1}
    \end{equation}
    Moreover, we compute that
    \begin{align*}
        \tilde \calL \wl 
        & = -c^2_LJ(\one_B x_2)  \\
        & = -c^2_L\one_B\nabla^\perp\psilm\cdot \nabla x_2-c^2_Lx_2\nabla^\perp\psilm\cdot\nabla(\one_B) \\
        & = -c^2_L\nabla^\perp \left( \one_B\psilm \right)\cdot \nabla x_2 +c^2_L  \psilm \nabla^\perp \one_B \cdot \nabla x_2 \\
        & = -c^2_L\nabla^\perp \left( \one_B\psilm \right)\cdot \nabla x_2 = \pa_{x_1}\wl,
    \end{align*}
    where we use $\nabla^\perp\psilm \perp \mathbf n$ on $\partial B$ and $\psilm|_{\pa B}=0$. Moreover, with the translation invariance of 2D Euler equation, we obtain that $\tilde \calL (\partial_{x_1} \wl) =0$ and thus $\tilde \calL^2\wl=0$.

    \noindent $\bullet$ As for $\pa_{x_2} \wl$, similar to the previous argument, observe that
    \begin{align*}
    \begin{split}
        \tilde L(\pa_{x_2}\wl) &= \pa_{x_2}\wl + c_L^2\one_B\Delta^{-1}(\pa_{x_2}\wl) \\
        &=\pa_{x_2}\wl+c^2_L\one_B\pa_{x_2}(\psilm-x_2)\\
        &=-c^2_L\one_B,
    \end{split}
    \end{align*}
    we know that $\pa_{x_2}\wl\notin\ker(\tilde L)$. Furthermore, using integration by parts and the fact that  $\wl|_{\pa B}=0$,
    \begin{equation}
        \scl{\tilde L(\pa_{x_2}\wl)}{\pa_{x_2}\wl} = -c^2_L\scl{\one_B}{\pa_{x_2}\wl} =0.
        \label{inner prod 2}
    \end{equation}
    We also compute that
    \begin{align*}
       \tilde  \calL (\pa_{x_2}\wl) = -c^2_LJ(\one_B) = -c^2_L\nabla^\perp\psilm\cdot\nabla(\one_B)=0. 
    \end{align*}
    Thus, $\pa_{x_2}\wl$ is a zero direction of $\scl{\tilde L\cdot}{\cdot}$ which lies in the generalized kernel of $\tilde \calL$ but not in the kernel of $\tilde L$.

     \noindent $\bullet$ As for $\mathbf x^\perp\cdot\nabla\wl$, since $\mathbf x^\perp\cdot\nabla =\pa_\theta$ and the Laplacian is rotation-invariant,
    \[
        [\mathbf x^\perp\cdot\nabla,\Delta^{-1}]=0,
    \]
    and it follows that
    \begin{align*}
    \begin{split}
        \tilde L(\mathbf x^\perp\cdot\nabla\wl)&= \mathbf x^\perp\cdot\nabla\wl+c^2_L\one_B\Delta^{-1}( \mathbf x^\perp\cdot\nabla\wl)\\
        &=\mathbf x^\perp\cdot\nabla\wl+c^2_L\one_B \mathbf x^\perp\cdot\nabla(\psilm-x_2)\\
        &=-c^2_L\one_B \mathbf x^\perp\cdot\nabla x_2  = -c^2_L x_1\one_B.
    \end{split}
    \end{align*}
    Then it yields that
    \begin{align*}
    \begin{split}
        \tilde \calL(\mathbf x^\perp\cdot\nabla\wl) 
        &= -c^2_LJ(x_1\one_B) = -c^2_L\nabla^\perp\psilm\cdot\nabla(x_1\one_B) \\
        &= c^2_L\one_B\pa_{x_2}\psilm = -\pa_{x_2}\wl,
    \end{split}
    \end{align*}
    and it holds that $\tilde \calL^2 (\mathbf x^\perp \cdot \na \wl) =0$ with the previous argument.
    Additionally, Since $\wl(\mathbf x) = -\frac{2c_LJ_1(c_Lr)}{J_0(c_L)}\sin(\theta)\one_B$ and $-\frac{2c_LJ_1(c_Lr)}{J_0(c_L)}\one_B\geq 0$,
    \begin{align}
        \scl{\tilde L(\mathbf x^\perp\cdot\nabla\wl)}{\mathbf x^\perp\cdot\nabla\wl} &=-c^2_L\scl{x_1\one_B}{\pa_\theta\wl} \notag \\
        &=-c^2_L\scl{r\cos(\theta)\one_B}{-\frac{2c_LJ_1(c_Lr)}{J_0(c_L)}\cos(\theta)} \notag \\
        &=\pi\int_0^1\frac{2c_L^3J_1(c_Lr)\,r}{J_0(c_L)}\;dr<0.
    \label{inner product 3}
    \end{align}
    Therefore, $\mathbf x^\perp\cdot\nabla\wl$ is another negative direction of $\scl{\tilde L\cdot}{\cdot}$ which lies in the generalized kernel of $\calL$ but not in the kernel of $\tilde L$. 

    \mbox{}
    
    In conclusion, combining all of the previous arguments concludes \eqref{algebraic relation} and \eqref{generalized ker of calL}. And it remains to prove \eqref{nonpositive of L: quadratic form}. Precisely, recalling \eqref{algebraic relation},
    \be
    \left( \tilde L (\pa_{x_2} \wl), \mathbf{x}^\perp \cdot \na \wl \right)
    =  \left( -c_L^2 \mathbf{1}_B x_2, \mathbf{x}^\perp \cdot \na \wl \right) =0.
    \label{inner prod 4}
    \ee
    Additionally, since $\pa_{x_2} \wl$ and $\mathbf x^\perp \cdot \na \wl$ are both even symmetric with respect to $x$-axis, and $\wl$ is odd symmetric with respect to $x$-axis, together with the fact that $\tilde L$ is invariant under the symmetry with respect to $x$-axis,
    \be
    \left( \tilde L (\pa_{x_2} \wl), \wl \right) =\left( \tilde L \left( \mathbf{x}^\perp \cdot \na  \wl \right), \wl \right) =0.
    \label{inner prod 5}
    \ee
    Consequently, combining \eqref{inner prod 1}, \eqref{inner prod 2}, \eqref{inner product 3}, \eqref{inner prod 4} and \eqref{inner prod 5}, for $g =a \wl + b  \partial_{2} \wl + c \mathbf{x}^\perp \cdot \na \wl$ with any $a,b,c \in \RR$,
    \begin{align*}
        & \quad \left( \tilde L g,g \right)
        = \left( \tilde L \left( a \wl + b  \partial_{2} \wl + c \mathbf{x}^\perp \cdot \na \wl \right),a \wl + b  \partial_{2} \wl + c \mathbf{x}^\perp \cdot \na \wl\right) \\
        & = a^2 \left( \tilde L \wl, \wl \right) + b^2 \left( \tilde L (\pa_{x_2} \wl), \pa_{x_2} \wl \right) + c^2 \left( \tilde L (\mathbf x^\perp \cdot \na \wl), \mathbf x^\perp \cdot \na \wl \right) \\
        & = a^2 \left( \tilde L \wl, \wl \right) + c^2 \left( \tilde L (\mathbf x^\perp \cdot \na \wl), \mathbf x^\perp \cdot \na \wl \right)  \le 0,
    \end{align*}
    which finishes the proof of \eqref{nonpositive of L: quadratic form}.
\end{proof}

\mbox{}
Based on the previous preliminaries, we are ready to prove Theorem \ref{main thm: linear} (\romannumeral1).
\begin{proof}[Proof of Theorem \ref{main thm: linear} (\romannumeral1)]
Recalling Corollary \ref{lemma: reduce eigenproblem to B} and Remark \ref{rmk: simplify to Hamiltonian form}, to study the unstable eigenmodes of $\calL$ \eqref{linearized operator: almost Hamiltonian form} on $L^2(\RR^2)$, it suffices to study $\tilde \calL$ \eqref{tilde calL} on $L^2(\RR^2)$. Precisely, by Lemma \ref{lemma: n-(L)} and Lemma \ref{lemma: non-positive directions of the bilinear form of L}, we obtain that $n^-(\tilde L) =3$ and $k^{\le 0}_0(\tilde \calL) \ge 3$. Since $\tilde \calL$ satisfies \ref{H1}, \ref{H2} and \ref{H3}, we plug this into the index formula \eqref{index theorem identity} to obtain that
\[
k_r(\tilde \calL )+2k_c(\tilde \calL )+2k_i^{\leq0}(\tilde \calL ) + 3 \le k_r(\tilde \calL )+2k_c(\tilde \calL )+2k_i^{\leq0}(\tilde \calL )+ k_0^{\leq0}(\tilde \calL )=n^{-}(\tilde L) =3,
\]
and this yields that
\[
k_r(\tilde \calL )= k_c(\tilde \calL ) =0,
\]
which in turn shows that
\[
k_r(\calL)= k_c(\calL )  =0.
\]
Consequently, we have finished the proof of Theorem \ref{main thm: linear}.
\end{proof}

\subsection{Spectral stability of $\calL$ on weighted $L^2$-spaces $X_\a$}\label{sec 25}
This subsection is devoted to the spectral stability of $\calL_\alpha$ on $X_\alpha$ for any fixed $\a>0$ (see \eqref{definition of weighted L2 spaces X_alpha} for its definition), so that we finish the proof of Theorem \ref{main thm: linear} (\romannumeral2). Firstly, recalling the "nearly" Hamiltonian form of $\calL_\a$ in \eqref{linearized operator: almost Hamiltonian form}, $\calL_\a$ can be written as
\be
\begin{aligned}
     \calL_\alpha = J_\a + \calK_\a,  \quad &  \text{ with }  \quad J_\a := \na^\perp \psilm \cdot \na : D(J_\a) \subset X_\a \to X_\alpha, \\
    & \text{ and } \quad   \calK_\a := c_L^2 \mathbf{1}_B \na^\perp \psilm \cdot \na \Delta^{-1},
    \end{aligned}
\label{decomp: L}
\ee
with the domain of $J_\a$ defined by 
\be
D(J_\a)= \{ f \in X_\alpha : Jf \in X_\a \text{ in  } \calD'(\RR^2) \} \subset X_\a.
\label{Dalpha(J)} \ee
Here we have used the fact that (see, for example, the derivations of \eqref{halmitonian 1} and \eqref{halmitonian 2})
\[
      c_L^2  \na^\perp \psilm \cdot \na (\mathbf{1}_B\Delta^{-1}\cdot)=c_L^2 \mathbf{1}_B \na^\perp \psilm \cdot \na \Delta^{-1}.
\]

\begin{lemma}[Compactness of $\calK_\alpha$ on $X_\a$]\label{Lemma: compactness of K}
    For any fixed $\alpha>0$, $\calK_\a: X_\a \to X_\a$ is a compact operator. 
\end{lemma}
\begin{proof}
    For any fixed $\a>0$, we choose $X_\a$ as in \eqref{definition of weighted L2 spaces X_alpha}. Recalling \eqref{decomp: L}, for brevity, we denote
    \[
        \calK_\alpha :=\mathbf{v}\cdot\nabla\Delta^{-1},\quad \mathbf{v}:=  c_L^2 \mathbf{1}_B \na^\perp \psilm.
    \]
    Then, $\calK_\a$ admits the following explicit expression
    \[
      \calK_\alpha f( \mathbf x) = \int_{\RR^2} k( \mathbf x, \mathbf y)f( \mathbf y)\,d \mathbf y,\quad k( \mathbf x, \mathbf y):= \frac{1}{2\pi}\mathbf{v}( \mathbf x)\cdot\frac{ \mathbf x- \mathbf y}{| \mathbf x-  \mathbf y|^2},
    \]
    Since $\mathbf{v}$ is bounded and $\text{supp}(\mathbf{v})\subset B$, there exists a uniform constant $C>0$ such that
    \[
        |k( \mathbf x,  \mathbf y)|\leq C\mathbf{1}_B( \mathbf x)\frac{1}{| \mathbf x-  \mathbf y|}.
    \]
    Hence for any $q > \max \big\{ 2,  \frac{2}{\a} \big\}$, by the Hardy-Littlewood-Sobolev inequality and H\"older's inequality, it holds for any $f\in X_\alpha$ that
    \begin{align}
     \| \calK_\alpha f \|_{L^q(\RR^2)}
    & \le C \Big\| \int_{\RR^2} \frac{f( \mathbf y)}{| \mathbf x- \mathbf y|} d \mathbf y \Big\|_{L^q(\RR^2)}
    \le  C \| f\|_{L^\frac{2q}{q+2}(\RR^2)} \notag \\
    & \le C \| \la  \mathbf x\ra^\a f \|_{L^2(\RR^2)} \| \la  \mathbf x \ra^{-\alpha} \|_{L^q(\RR^2)} \le C   \| \la  \mathbf x\ra^\a f \|_{L^2(\RR^2)}.
    \label{HLS: weighted version}
    \end{align}
    Since $\calK_\alpha f$ is compactly supported in $B$, by H\"older's inequality we have
    \[
    \| \la  \mathbf x \ra^\alpha \calK_\alpha f \|_{L^2(\RR^2)}
    \le C \|\calK_\alpha f\|_{L^2(B)}\le C\|\calK_\alpha f\|_{L^q(B)}\le C \big\| \la  \mathbf x \ra^\alpha f \big\|_{L^2(\RR^2)},
    \]
    hence $\calK_\alpha: X_\alpha \to X_\alpha$ is a bounded operator.

    Next, we prove the compactness of $\calK_\alpha: X_\a \to X_\a$. Firstly, observe that there exists $\mathbf {v}_n \in C_c^\infty(\bar B)$ with $\text{supp}(\mathbf{v}_n)\subset B$, such that
    \[
    \calK_{\a,n} := \mathbf{v}_n \cdot \na \Delta^{-1}: X_\a \to X_\a \quad \text{ satisfying} \quad \| \calK_{\a,n} - \calK_\alpha\|_{X_\a \to X_\alpha} \to 0 \text{ as }  n \to \infty.
    \]
    By the closedness of the set of compact operators under the operator norm, it suffices to prove that $\calK_{\a,n}: X_\a \to X_\a$ is compact for each $n$. In fact, for any $\{ f_m \} \subset X_\alpha$ with $\sup_n \| \la \mathbf x \ra^\alpha f_m \|_{L^2(\RR^2)} \le 1$, similar to \eqref{HLS: weighted version}, we have
    \[
    \| \calK_{\a, n} f_m \|_{L^2(\RR^2)} \le C_n \| \la \mathbf x \ra^\alpha f_m \|_{L^2(\RR^2)} \le C_n, \text{ for all } n \in \ZZ_+,
    \]
    and by the boundedness of  $\nabla^{2}\Delta^{-1}:L^2(\RR^2)\to L^2(\RR^2)$ as a Calder\'on-Zygmund operator,
    \begin{align*}
    \| \calK_{\a, n} f_m \|_{\dot{H}^1(\RR^2)}
    &\le \|\nabla\mathbf{v}_n\|_{L^\infty(\RR^2)}\|\one_B\nabla\Delta^{-1}f\|_{L^2(\RR^2)}+\| \mathbf {v}_n \|_{L^\infty} \|\na^2 \Delta^{-1} f_m \|_{L^2(\RR^2)}\\
    &\le C_n \dr{\| \la \mathbf x \ra^\alpha f_m \|_{L^2(\RR^2)} + \|\na^2 \Delta^{-1} f_m \|_{L^2(\RR^2)}} \\
    &\le C_n, \text{ for all } n \in \ZZ_+.
    \end{align*}
    Noting that $\text{supp } (\calK_{\a,n} f_m) \subset B$ for each $m \ge 1$, using the Rellich-Kondrachov compactness theorem, $\{ \calK_{\a,n} f_m \}_m$ is precompact in $L^2(\RR^2)$. Additionally, since $\| \cdot \|_{X_\a} \simeq_\alpha \| \cdot \|_{L^2}$ for any $L^2$ function compactly supported on $B$, $\{ \calK_{\a,n} f_m \}_{m}$ is precompact in $X_\a$, and thus we have concluded the proof.
\end{proof}
Next, we study the spectrum of $J_\alpha$ on $X_\a$. 

\begin{lemma}[Spectrum of transport operator]\label{lemma: spectrum of the transport op on weighted L2 is on the imaginary axis}
Let $J_\a$ for $\a > 0$ defined as in \eqref{decomp: L}. 
Then $J_\a$ is a closed operator and $\sigma(J_\a) \subset i\RR$. 
\end{lemma}

\begin{proof}
\noindent \textbf{Step 1. Global-in-time existence of measure-preserving flow map.} Since $\wl \in C^{0, 1}(\RR^2)$ and is supported in $B$ from its definition \eqref{Lamb dipole: vorticity}, we have $\mathbf b := \nabla^\perp \psilm = \nabla^\perp (\Delta_{\RR^2}^{-1}\wl + x_2) \in C^1(\RR^2)$ by standard elliptic regularity and $\nabla \cdot \mathbf b = 0$.

Let $\Phi_t$ is the flow map generated by the incompressible velocity fields $\mathbf{b}$, i.e.
\[
\frac{d}{dt} \Phi_t (\mathbf x) = \mathbf b (\Phi_t(\mathbf x)), \; \text{ with } \; \Phi_0(\mathbf x) = \mathbf x \in \RR^2.
\]
Then we have
\[
\frac{d}{dt}\langle \Phi_t(\mathbf x) \rangle
=
\frac{\Phi_t(\mathbf x) \cdot \mathbf b(\Phi_t(\mathbf x))}{\langle \mathbf \Phi_t(\mathbf x)\rangle},
\]
which, together with the fact that $\| \mathbf b \|_{L^\infty} < \infty$, induces that
\be
\la \Phi_t (\mathbf x) \ra \le \la \mathbf x \ra + C(\| \mathbf b \|_{L^\infty}) |t| \le C(\| \mathbf b \|_{L^\infty}) \la t \ra \la \mathbf x \ra, 
\label{growth rate}
\ee
so that we obtain a global-in-time flow map $\Phi_t$ induced by $\mathbf b$. Together with incompressibility of $\mathbf b$, the flow map $\Phi_t$ preserves the Lebesgue measure in the sense that
\be
\det D\Phi_t(\mathbf x)=1.
\label{measure preseving}
\ee

\mbox{}

\noindent \textbf{Step 2. $J_\a$ generates a continuous semigroup on $X_\alpha$.} For any $f \in X_\a$,
\[
U(t)f:=f\circ\Phi_t, \quad \text{ for all } t \in \RR,
\]
where $\Phi_t$ is the flow map generated by $\mathbf b$ mentioned in \textit{Step 1}. By changing variables $y=\Phi_t(\mathbf x)$, together with \eqref{measure preseving},
\[
\|U(t)f\|_{X_\alpha}^2
=
\int_{\mathbb R^2}
\langle \mathbf x \rangle^{2\alpha}
|f(\Phi_t(\mathbf x))|^2\,d \mathbf x
= \int_{\mathbb R^2}
\langle \Phi_{-t}(\mathbf y)\rangle^{2\alpha}
|f(\mathbf y)|^2\,d\mathbf y.
\]
Recalling \eqref{growth rate}, this yields that
\be
\| U(t) f \|_{X_\a}^2
= \int_{\mathbb R^2}
\langle \Phi_{-t}(\mathbf y)\rangle^{2\alpha}
|f(\mathbf y)|^2\,d\mathbf y
\le C \la t \ra^{2 \a} \la \mathbf y \ra^{2 \alpha} |f(\mathbf y)|^2 d \mathbf y
\le C \la t \ra^{2 \a} \| \la \cdot \ra^\alpha f \|_{L^2}^2,
\label{semigroup estimate}
\ee
and this means that $\{ U(t) \}_{t \ge0}$ is a bounded semigroup on $X_\alpha$. Additionally, by the density of $C_c^\infty(\mathbb R^2)$ in $X_\alpha$ and the local uniform boundedness of $U(t)$, it is a strongly continuous semigroup on $X_\alpha$. In particular, let $A$ be the generator of $\{U(t)\}_{t \ge 0}$:
\[
D(A)
=
\left\{
f\in X_\alpha:\;
\lim_{t\to0}\frac{U(t)f-f}{t}
\text{ exists in }X_\alpha
\right\}.
\]
We now show that $A$ agrees with the maximal distributional realization of $J_\a$ with domain $D(J_\a)$ determined in \eqref{Dalpha(J)}.

On the one hand, suppose $f\in D(A)$ and $Af=g$. For every $\varphi\in C_c^\infty(\mathbb R^2)$, recalling \eqref{measure preseving},
\begin{align*}
& \quad \int_{\mathbb R^2} g\varphi\,d\mathbf x
= \lim_{t\to0}
\int_{\mathbb R^2}
\frac{f(\Phi_t(\mathbf x))-f(\mathbf x)}{t}\varphi(\mathbf x)\,d\mathbf x
 \\
 &= \lim_{t\to0}
\int_{\mathbb R^2}
f(\mathbf y)
\frac{\varphi(\Phi_{-t}(\mathbf y))-\varphi(\mathbf y)}{t}
\,d\mathbf y 
 = \int_{\RR^2} f(\mathbf y) \left( - \mathbf b \cdot \na \varphi \right)(\mathbf y) d \mathbf y,
\end{align*}
and this implies that $\mathbf b\cdot\nabla f=g $ in $\mathcal D'(\mathbb R^2)$, so $f\in D(J_\a)$ and $J_\a f=Af$.

On the other hand, suppose $f\in D(J_\a)$ and set
\[
g:=J_\a f=\mathbf b\cdot\nabla f\in X_\alpha.
\]
We claim that the following holds distributionally:
\[
U(t)f-f
=
\int_0^t U(s)g\,ds.
\]
In fact, for any chosen $\varphi\in C_c^\infty(\mathbb R^2)$, reaclling \eqref{measure preseving},
\begin{align*}
\frac{d}{dt}
\int_{\mathbb R^2} U(t)f\,\varphi\,d\mathbf x
& =
\frac{d}{dt}
\int_{\mathbb R^2} f(\mathbf y)\varphi(\Phi_{-t}(\mathbf y))\,d\mathbf y 
 = \int_{\RR^2} f(\mathbf y) \left( - \mathbf b (\mathbf y) \cdot \na (\varphi \circ \Phi_{-t}) (\mathbf y) \right) d\mathbf y \\
& = \int_{\RR^2} g(\mathbf y) \varphi(\Phi_{-t}(\mathbf y) ) d \mathbf y
= \int_{\RR^2} g(\Phi_t(\mathbf x)) \varphi(\mathbf x) d \mathbf x = \int_{\RR^2} U(t) g \varphi \ d \mathbf x, 
\end{align*}
Using the identity
\[
\frac{d}{dt}\varphi(\Phi_{-t}(\mathbf y))
=
-\mathbf b(\mathbf y)\cdot\nabla\bigl(\varphi\circ\Phi_{-t}\bigr)(\mathbf y),
\]
and the distributional relation $\mathbf b\cdot\nabla f=g$, we get
\[
\frac{d}{dt}
\int_{\mathbb R^2} U(t)f\,\varphi\,d\mathbf x
=
\int_{\mathbb R^2}
g(\mathbf y)\varphi(\Phi_{-t}(\mathbf y))\,d\mathbf y
=
\int_{\mathbb R^2}
U(t)g\,\varphi\,d\mathbf x,
\]
and thus this implies that
\[
U(t)f-f
=
\int_0^t U(s)g\,ds, \quad \text{ in } \mathcal D'(\mathbb R^2).
\]
Additionally, since both sides belong to $X_\alpha$, the identity is still valid on $X_\alpha$. Dividing by $t$ and then strong continuity of $U(t)$, we obtain that the generator should be
\[
\frac{U(t)f-f}{t}
=
\frac1t\int_0^t U(s)g\,ds \to g \; \text{ as }  t \to 0, \text{ in  the sense of } X_\a.
\]
Consequently, we conclude that $f\in D(A)$ and $Af=g=J_\a f$. In summary, we have obtained that $A=J_\a$.

\mbox{}

\noindent \textbf{Step 3. Spectrum of $J_\a$.} It remains to prove the spectral inclusion. Recalling \eqref{semigroup estimate}, for every $\varepsilon>0$, there exists $C(\varepsilon)>0$ such that
\[
\|U(t)\|_{X_\alpha\to X_\alpha}
\leq C(\ep) e^{\varepsilon |t|}.
\]
For $\lambda \in \{z \in \CC: \operatorname{Re}z>\varepsilon \}$, the usual Laplace transform formula gives
\[
(\lambda-J_\a)^{-1}f
=
\int_0^\infty e^{-\lambda t}U(t)f\,dt,
\]
and thus $\lambda\in\rho(J_\a)$. Similarly, using the semigroup for negative times,
\[
\{z \in \CC: \operatorname{Re}z< -\varepsilon \} \in\rho(J_\a).
\]
With the arbitrariness of $\varepsilon>0$, this implies that $\sigma(J_\a)\subset i\mathbb R$.
\end{proof}

Next, we introduce the following general version of compact perturbation theory. 
\begin{lemma}[Spectrum of compactly-perturbed operator]\label{Lemma: spectrum of compactly-perturbed operator}
    If $J: D(J) \subset X \to X$ is a closed, densely defined operator on a Banach space $X$, and $K: X \to X$ is compact, then we have
    \[
    \Big(\sigma_c(J+K) \cup \sigma_r(J+K) \Big) \subset \sigma(J).
    \]
\end{lemma}

\begin{proof}
    We prove by contradiction. Assuming there exists $\lambda \in \Big(\sigma_c(J+K) \cup \sigma_r(J+K) \Big) \cap \rho(J)$, we know that $(\lambda I - J)^{-1}: X \to X$ is a bounded linear operator, so that $K(\lambda I -J)^{-1} : X \to X$ is compact. Then from the decomposition
    \[
    \lambda I - (J+K) = \lambda I - J - K = \left( I - K(\lambda I -J)^{-1} \right) \left( \lambda I -J \right),
    \]
    we claim that $I - K(\lambda I -J)^{-1}$ is invertible, then we know that $\lambda I - (J+K): X \to X$ is invertible, so that $\lambda \not\in \sigma_c(J+K) \cup \sigma_r(J+K)$.
    
     To prove the claim, by the Fredholm alternative, it is equivalent to showing that
    \[
    \ker \left(  I - K(\lambda I -J)^{-1} \right) = \{ 0 \}.
    \]
    In fact, if there exists $0 \not= g \in \ker \left(  I - K(\lambda I -J)^{-1} \right) $, then
    \[
    g - K(\lambda I -J)^{-1} g =0, \quad \Rightarrow \quad (J+K) u = \lambda u, \text{ with }  u = (\lambda I - J)^{-1} g \in X,
    \]
    so that $\lambda \in \sigma_p(J+K)$, which contradicts $\l \in \sigma_c(J+K) \cup \sigma_r(J+K)$.
\end{proof}

With the preliminaries given by Lemma \ref{Lemma: compactness of K}, Lemma \ref{lemma: spectrum of the transport op on weighted L2 is on the imaginary axis} and Lemma \ref{Lemma: spectrum of compactly-perturbed operator}, we are ready to finish the proof of Theorem \ref{main thm: linear} (\romannumeral2).
\begin{proof}[Proof of Theorem \ref{main thm: linear} (\romannumeral2)]
As for $J_\a: D(J_\a) \subset X_\a \to X_\a$ given in \eqref{decomp: L}, since it satisfies the conditions of Lemma \ref{lemma: spectrum of the transport op on weighted L2 is on the imaginary axis}, $J_\a: D(J_\a) \subset X_\a \to X_\a$ is a closed, densely defined operator on $X_\a$, and $\sigma(J_\a) \subset i \RR$. Then combining Lemma \ref{Lemma: compactness of K}, Lemma \ref{Lemma: spectrum of compactly-perturbed operator} and the decomposition of $\calL_\a$ \eqref{decomp: L}, we conclude that
\[
\left( \sigma_{c} (\calL_\a) \cup \sigma_r(\calL_\a) \right) \subset \sigma(J_\a) \subset i \RR,
\]
and it suffices to study the $\sigma_p(\calL_\a)$. Since $X_\alpha\subset L^2(\RR^2)$ by definition, applying Theorem \ref{main thm: linear} (\romannumeral1), we have
    \[
       \Big( \sigma_p(\calL_\alpha)\cap \{\lambda\in\mathbb C:\Re\lambda<0\} \Big) 
    \subset \Big( \sigma_p \left(\calL \right )\cap \{\lambda\in\mathbb C:\Re\lambda<0\} \Big) =\varnothing.
    \]
   In conclusion, we obtain the spectral stability of $\calL$ in $X_\a$ in the sense that
   \[
   \sigma(\calL_\a) \cap \big\{ \lambda \in \CC: \Re \lambda <0 \big\} = \emptyset,
   \]
   and thus we have finished the proof of Theorem \ref{main thm: linear} (ii).
\end{proof}

\section{Quantitative orbital stability of Lamb-Chaplygin dipole}
\label{sec3}

This section is devoted to the quantitative orbital stability of Lamb-Chaplygin in the class $\calX_{\text{odd},+}$ \eqref{Xodd+}. First, note that for any $\omega$ satisfying the odd symmetry assumption with respect to $x$-axis\footnote{In later discussion in the section, out of the simplification of notation, we denote $L^2{(\RR_+^2)}$ the collection of all $L^2(\RR^2)$ functions that are all odd symmetric with respect to the $x$-axis.}, the dynamical information can be reduced to the upper half-plane $\RR_+^2$, and the kinetic energy, enstrophy and the impulse can be respectively rewritten as
\begin{align*}
      E[\omega] &= -\frac{1}{2} \int_{\RR_+^2} \omega(\mathbf x) \psi(\mathbf x) d \mathbf x = - \frac{1}{2} \int_{\RR_+^2} \omega(\mathbf x) \left( \Delta_{\RR_+^2}^{-1} \omega \right)(\mathbf x)  d \mathbf x \\
     & = - \frac{1}{8 \pi} \iint_{\RR_+^2 \times \RR_+^2} \log \left( 1+ \frac{4 x_2 y_2}{|\mathbf x - \mathbf y|^2} \right) \omega (\mathbf x) \omega(\mathbf y) d \mathbf x d \mathbf y, \\
      K[\omega] &= \| \omega \|_{L^2(\RR_+^2)}^2, \\
     I[\omega] &= \int_{\RR_+^2} x_2 \omega(\mathbf x) d \mathbf x.
\end{align*}
Then in the class $\calX_{\text{odd},+}$, we define the Lagrangian functional as 
    \be
\calF[\omega] = \frac{1}{2 c_L^2}\| \omega \|_{L^2(\RR_+^2)}^2 -  E[\omega] + \int_{\RR_+^2} x_2 \omega(\mathbf x) d\mathbf x.
\label{functional: def}
\ee
Via direct calculation, the first-order variation and second-order variation of $\calF$ at $\omega = \wl$ can be respectively written as
\begin{align}
\Big\la \frac{\delta \calF}{\delta \omega}, h \Big\ra &= \frac{1}{c_L^2} \int_{\RR_+^2} \errl (\mathbf x) h (\mathbf x) d \mathbf x, 
\label{first variation} \\
\Big\la \frac{\delta^2 \calF}{\delta^2 \omega} h, h  \Big\ra &= 2 \left(\calS h, h \right), \quad \calS := \frac{1}{2c_L^2} I - \frac{1}{2} \left( -\Delta_{\RR_+^2} \right)^{-1},
\label{second variation}
\end{align}
where $\errl$ is defined from \eqref{error: def}.
In particular, we denote $\calQ[h]$ as the difference functional
\be
\calQ[h] := \calF[\wl + h] - \calF[\wl] = \frac{1}{c_L^2} \int_{\RR_+^2} \errl(\mathbf x) h(\mathbf x) d \mathbf x + \left( \calS h, h \right).
\label{def: Q}
\ee

In Section \ref{sec 32} and Section \ref{sec 33}, we will prove the coercivity of $\calQ$ restricted to a certain admissible set. Afterwards, we introduce modulation analysis and the full Lyapunov functional in Section \ref{sec 34} to conclude the quantitative orbital stability Theorem \ref{thm: orbital stability sec3}.

\vspace{0.1cm}

\subsection{Coercivity of $\calQ$ in interior region}
\label{sec 32}

In this subsection, we analyze the coercivity of the following linear operator 
\be
\tilde \calS = \frac{1}{2c_L^2}Id - \frac{1}{2} \mathbf{1}_B (-\Delta_{\RR_+^2})^{-1} \mathbf{1}_B.
\label{linear op: cutoff version}
\ee
The corresponding quadratic form matches with $\calQ$ restricted in the disk: 
\[ \calQ[h] = (\tilde \calS h, h)\quad {\rm if}\,\,\,\,\supp h \subset B. \]

Notice that $\tilde \calS$ resembles $\tilde L$ \eqref{tilde calL}. By mimicking the proof of Lemma \ref{lemma: n-(L)}, we can characterize its center-unstable eigen directions and obtain a coercivity estimate from the spectral gap. 

\begin{lemma}[Spectrum of $\tilde \calS$ on $L^2(\RR_+^2)$]
\label{lem: tilde S spectrum gap}
    As for $\tilde\calS: L^2(\RR_+^2) \to L^2(\RR_+^2)$ defined in \eqref{linear op: cutoff version}, we have 
    \be  \sigma \left( \tilde \calS\big|_{L^2(\RR^2_+)} \right) \cap (-\infty, 0] = \{ \l_1, 0\},\quad \text{ where }\,\, \lambda_1 = \frac{1}{2c_L^2}\left( 1-\frac{\mu_{1,1}^2}{\mu_{0,1}^2}  \right) <0,  \label{smallest eigenvalue} \ee
    and both eigenvalues are simple. More specifically,
\be
\ker\left(\tilde \calS \big|_{L^2(\RR^2_+)}- \l_1\right) = {\rm span} \{g_1\}, \quad \ker\left(\tilde \calS \big|_{L^2(\RR^2_+)}\right) = {\rm span} \{\pa_{x_1} \wl\}, 
\label{smallest eigenfunction}
\ee
where 
\be 
g_1(\mathbf x) =  \mathbf{1}_B(\mathbf x) J_{1}(\mu_{0,1} r) \sin \theta. 
\ee 
Moreover, we have the spectral gap estimate: there exists $c_0 > 0$ such that 
\be
(\tilde \calS  f,f )_{L^2(\RR_+^2)} \ge c_0 \|f\|_{L^2(\RR_+^2)}^2,\quad  \text{ for any } f \perp \partial_{x_1} \wl, g_1.
\label{coercivity: original version}
\ee
\end{lemma}

\begin{proof}
    The proof of Lemma \ref{lem: tilde S spectrum gap} is very similar to Lemma \ref{lemma: n-(L)}. But out of completeness, let us sketch the proof. Firstly, it is easy to check that $\tilde \calS$ is invariant and self-adjoint on $L^2(\RR_+^2)$. Next, $\mathbf{1}_B (-\Delta_{\RR_+^2})^{-1} \mathbf{1}_B$ is non-negative and Hilbert-Schmidt operator in view of \eqref{Hilbert-Schmidt of tilde L} and hence compact on $L^2(\RR_+^2)$ \cite[Section 6.8.5]{helffer2013spectral} and \cite[Theorem 9.4]{zbMATH00837307}, and thus all elements in $\sigma\left( \tilde \calS \right) \setminus \big\{ \frac{1}{2c_L^2} \big\}$ are discrete \cite[Theorem 6.8]{MR2759829}. In particular, 
 \[\sigma(\tilde \calS) \cap \{ \lambda: \text{Re } \lambda \le 0 \} =\sigma_{\text{disc}}(\tilde \calS) \cap \{ \lambda: \text{Re } \lambda \le 0  \}.
 \]Then the spectral gap property of $\tilde \calS$ given in \eqref{coercivity: original version} is a result of \eqref{smallest eigenvalue} and \eqref{smallest eigenfunction}.

 To study \eqref{smallest eigenvalue} and \eqref{smallest eigenfunction}, we consider the eigen equation $\tilde \calS g = \lambda g$ with $\lambda \le 0$, then we see that $g$ solves
 \[
 -\mathbf{1}_B (-\Delta_{\RR_+^2})^{-1} \mathbf{1}_B g = \left( 2\lambda - \frac{1}{c_L^2} \right) g, \quad \Rightarrow \quad \text{supp }g \subset B.
 \]
 Then with a similar argument as the proof of Lemma \ref{lemma: n-(L)}, we obtain that the related eigenvalues in $L^2(\RR_+^2)$ should be
 \[
 \lambda_{m,j} = \frac{1}{2c_L^2} \left( 1 - \frac{\mu_{1,1}^2}{\mu_{m-1,j}^2} \right), \text{ with  the eigenfunction } g_{m,j}(\mathbf x) = \mathbf{1}_B (\mathbf x) J_m\left( \mu_{m-1,j} r \right) \sin m \theta.
 \]
 for any $m,j \in \ZZ_{+}$. In particular, the only scenario leading to the negativity and vanishing of $\lambda_{m,j}$ is $(m,j)=(1,1)$ and $(2,1)$ respectively. Thus, we have finished the proof.
\end{proof}

\begin{proposition}[Coercivity of $\tilde \calS$ \eqref{linear op: cutoff version}]
\label{prop: calS coercivity}
     As for $\tilde \calS: L^2(\RR_+^2) \to L^2(\RR_+^2)$ defined in \eqref{linear op: cutoff version}, there is a uniform constant $a>0$ such that
     \begin{align}
          \left( \tilde \calS h, h \right)_{L^2(\RR_+^2)} \ge a \| h \|_{L^2(\RR_+^2)}^2,  
        \quad & \text{ for all } h \in L^2(\RR_+^2), \;h \perp_{L^2(\RR_+^2)} \mathbf{1}_B x_2, \notag \\
        & \quad \text{ and } h \perp_{L^2(\RR_+^2)} \partial_{x_1} \wl.
        \label{spectral gap: version 1}
     \end{align}
\end{proposition}

\begin{remark}
While the orthogonality conditions in \eqref{coercivity: original version} seem also sufficient for the subsequent modulation argument, this new set of orthogonal conditions is simpler and more natural, and provides additional flexibility to future application of this coercivity result. 
\end{remark}

\begin{proof}[Proof of Proposition \ref{prop: calS coercivity}] Let us introduce the variational problem
\begin{align}
    a := \inf_{h \in \calB} \left( \tilde \calS h, h \right)_{L^2(\RR_+^2)}, \quad  
    & \text{ with } \calB := \Big\{ h \in L^2(\RR_+^2): h \perp_{L^2(\RR_+^2)} \mathbf{1}_B x_2, \notag  \\
    & \qquad \qquad \qquad h \perp_{L^2(\RR_+^2)} \partial_{x_1} \wl \text{ and } \| h\|_{{L^2(\RR_+^2)}} =1 \Big\}.
\label{variational: min-max}
\end{align}
We will prove $a > 0$ by contradiction. From now on, we suppose $a \le 0$ and will derive a contradiction. From \eqref{smallest eigenvalue}, we have $a \ge \l_1$.

\mbox{}

\noindent \textbf{Step 1. Existence of minimizer.}
We firstly claim there exists a minimizer $h_* \in \calB$ such that $a = (\tilde \calS h_*, h_*)_{L^2(\RR^2_+)} $. 

Indeed, pick a minimizing sequence $\{ h_n \} \subset \calB$ such that
\[
\left( \tilde \calS h_n, h_n \right) \to a,\quad  \text{ as } n \to \infty.
\]
Note that $\| h_n \|_{L^2(\RR_+^2)} =1$ for all $n \ge 1$, there exists a subsequence (we still denote it by $\{ h_n\}$) such that $h_n \rightharpoonup h_*$ weakly in $L^2(\RR_+^2)$ for some $h_* \in L^2(\RR_+^2)$. In particular, $h_*$ inherits the othogonality conditions $h_* \perp_{L^2(\RR_+^2)} \mathbf{1}_B x_2,  \partial_{x_1} \wl$, and satisfies $\| h_* \|_{L^2(\RR^2_+)} \le 1$. Note that $\mathbf{1}_B (-\Delta)^{-1} \mathbf{1}_B: L^2(\RR_+^2) \to L^2(\RR_+^2)$ is a compact operator, this yields that
\[
\mathbf{1}_B (-\Delta)^{-1} \mathbf{1}_B h_n \to \mathbf{1}_B (-\Delta)^{-1} \mathbf{1}_B h_* \quad \text{ strongly in } L^2(\RR_+^2) \text{ as } n \to \infty,
\]
and thus 
\begin{align*}  (\mathbf{1}_B (-\Delta)^{-1} \mathbf{1}_B h_*, h_*)_{L^2(\RR^2_+)} &= \lim_{n \to \infty} (\mathbf{1}_B (-\Delta)^{-1} \mathbf{1}_B h_n, h_n)_{L^2(\RR^2_+)} \\
&= \lim_{n \to \infty} \left( \frac{1}{c_L^2} \| h_n \|_{L^2}^2 - 2(\tilde S h_n, h_n) \right)_{L^2(\RR^2_+)}  = \frac{1}{c_L^2} - 2a \ge \frac{1}{c_L^2},
\end{align*}
where we exploited the contradiction assumption $a\le 0$. This implies that $h_* \neq 0$.

Next, by semi-lower continuity of weak convergence \cite[Proposition 3.5]{MR2759829}, we obtain that
\be
\left( \tilde{\calS} h_*, h_* \right)_{L^2(\RR_+^2)}
\le \liminf_{n \to \infty} \left( \tilde{\calS} h_n, h_n \right)_{L^2(\RR_+^2)} = a. \label{eqcontradiction1}
\ee
Moreover, by the definition of $a$ \eqref{variational: min-max}, we have
\be \| h_* \|_{L^2(\RR_+^2)}^{-2} \left( \tilde{\calS} h_*, h_* \right)_{L^2(\RR_+^2)} \ge a. \label{eqcontradiction2} \ee 
Then we can conclude the claim by the following discussion:

(1) If $a = 0$, these inequalities implies $\left( \tilde{\calS} h_*, h_* \right)_{L^2(\RR_+^2)} = 0$, so we can replace $h_*$ by $\tilde h_* = h_* / \| h_* \|_{L^2(\RR_+^2)} \in \calB$ which still satisfies $\left( \tilde{\calS} \tilde h_*, \tilde h_* \right)_{L^2(\RR_+^2)} = 0$.

(2) if $a < 0$, these inequalities indicates that $\| h_* \|_{L^2(\RR_+^2)}^{-2} \le 1$. Combined with $\| h_*\|_{L^2(\RR_+^2)} \le 1$, we obtain $\| h_* \|_{L^2(\RR_+^2)} = 1$ and thus $h_* \in \calB$. Then the inequalities in \eqref{eqcontradiction1}-\eqref{eqcontradiction2} becomes equalities, leading to $\left( \tilde{\calS} h_*, h_* \right)_{L^2(\RR_+^2)} = a$.

\mbox{}

\noindent \textbf{Step 2. End of proof.}
We can write the Euler-Lagrange equation for the minimizer $h_* \in \calB$ as
\[ 
\tilde \calS h_* = a h_* + b \mathbf{1}_B x_2 + c \partial_{x_1} \wl,
\]
where $a \le 0$ is as \eqref{variational: min-max}. Taking $L^2$ inner product with $\partial_{x_1} \wl$ onto both sides, with $h_* \in \calB$, we obtain that $c=0$ and hence
\[
\tilde \calS h_* = a h_* + b \mathbf{1}_B x_2.
\]
Then we will derive a contradiction from this identity by discussing the following three cases. That concludes the proof of this proposition.

\noindent (1) If we assume that $a =0$, then $\tilde \calS h_* = b \mathbf{1}_B x_2$. Recall \eqref{error: def}, we notice that
\begin{align}
2c_L^2 \tilde \calS \wl 
& = \wl + c_L^2 \mathbf{1}_B \psil = \mathbf{1}_B \left( \wl + c_L^2  \psil \right) \notag\\
& = \mathbf{1}_B \left( \errl - c_L^2 x_2 \right) = -c_L^2 \mathbf{1}_B x_2,
\label{F(0)} 
\end{align}
and it implies that $h_* = - b \left( \frac{1}{2} \wl + d_1 \partial_{x_1} \wl \right) $ for some $d_1 \in \RR$. Thus by \eqref{energy+impulse+enstrophy of lamb dipole},
\[
0 = b\left( h_*, \mathbf{1}_B x_2 \right)_{L^2(\RR_+^2)} =  - \frac{b}{2}\left( \mathbf{1}_B x_2 , \wl \right)_{L^2(\RR_+^2)} = - \frac{b}{2} I[\wl], \quad \Rightarrow \quad b =0.
\]
Hence $\tilde \calS h_* = 0$, and thus $h_* = c' \partial_{x_1} \wl$ for some $c' \not= 0$, which contradicts $h_* \perp \partial_{x_1} \wl$.

\noindent (2) If we assume that $a \in  \left( \lambda_1,0 \right) \subset \rho(\tilde \calS)$,
\[
h_* = b \left( \tilde \calS -a \right)^{-1} \mathbf{1}_B x_2, \quad \Rightarrow \quad 0 = (h_*, \mathbf{1}_B x_2)_{L^2(\RR_+^2)} = b \left( \left( \tilde \calS -a \right)^{-1} \mathbf{1}_B x_2, \mathbf{1}_B x_2\right)_{L^2(\RR_+^2)}.
\]
We define for $t \in (\l_1, 0)$ that 
\[
F(t):= \left( \left( \tilde \calS - t \right)^{-1} \mathbf{1}_B x_2, \mathbf{1}_B x_2\right)_{L^2(\RR_+^2)} \quad \Rightarrow \quad F'(t) = \Big\| \left( \tilde \calS -t \right)^{-1} \mathbf{1}_B x_2 \Big\|_{L^2(\RR_+^2)}^2 \ge 0.
\]
Recalling \eqref{F(0)}, we find that $F(0)= - \frac{1}{2} \left( \mathbf{1}_B x_2, \wl \right)= - \frac{1}{2} I[\wl]<0$, and integrating $F'(t)$ on $t \in [a, 0]$ yields $F(a) <0$. This implies that $b=0$ and hence $\tilde \calS h_* = a h_*$, contradicting to $a \in \left( \lambda_1, 0 \right) \in \rho(\tilde \calS)$.

\noindent (3) If we assume that $a = \lambda_1$, by min-max principle, $h_*$ should be the unique minimizer of the variational problem
\[
\inf_{\| f \|_{L^2(\RR_+^2)}=1, \text{ $f$ is odd in $x$-axis}} \left( \tilde \calS f, f \right)_{L^2(\RR_+^2)},
\]
thus $h_* =c'' \mathbf{1}_B J_{1} \left( \mu_{0,1} r \right) \sin \theta$ for some $c'' \not= 0$. However, by using Bessel identity \eqref{Bessel identity},
\begin{align*}
    & \quad  \left( \mathbf{1}_B J_{1} \left( \mu_{0,1} r \right) \sin \theta, \mathbf{1}_B x_2 \right)_{L^2(\RR_+^2)} 
     = \left( \mathbf{1}_B J_{1} \left( \mu_{0,1} r \right) \sin \theta, \mathbf{1}_B r \sin \theta \right)_{L^2(\RR_+^2)} \\
     & = \left( \int_0^1 J_1(\mu_{0,1} r) r^2 dr \right) \left( \int_0^\pi \sin^2 \theta d \theta \right)
     = \frac{\pi}{\mu_{0,1}^3} \int_0^{\mu_{0,1}}  r^2 J_1(r) dr \\
     & = \frac{\pi}{\mu_{0,1}^3} \int_0^{\mu_{0,1}} \frac{d}{dr} \left( r^2 J_2(r) \right) dr 
     = \frac{\pi}{\mu_{0,1}} J_2 \left( \mu_{0,1} \right) \not = 0,
\end{align*}
this contradicts $(h_*, \mathbf{1}_B x_2)_{L^2(\RR_+^2)}=0$. 
\end{proof}

Although Proposition \ref{prop: calS coercivity} provides a satisfactory coercivity estimate, the poor regularity of orthogonality directions $\mathbf{1}_B x_2 \in L^2(\RR_+^2)$ and $\partial_{x_1} \wl \in L^2(\RR_+^2)$ causes trouble for modulation analysis later (the same as the original version of conditions given in \eqref{coercivity: original version}). To address this issue, we replace them by smooth functions as their approximation in $L^2(\RR^2)$, which inherits the coercivity and non-degeneracy of $\tilde \calS$. 

\begin{corollary}[Modified spectral gap property for $\tilde \calS$]
\label{cor: calS coercivity}
There exists $(W_1,W_2) \in \left( C_c^\infty(\RR_+^2) \right)^2$ such that the following spectral gap for $\tilde \calS: L^2(\RR_+^2) \to L^2(\RR_+^2)$ holds:
    \be
    \left( \tilde \calS h ,h \right)_{L^2(\RR_+^2)} \ge \frac{a}{2} \| h \|_{L^2(\RR_+^2)}^2, \text{ for all } h \in L^2(\RR_+^2), \; h \perp W_1, \text{ and } h \perp W_2,
    \label{spectral gap: modified}
    \ee
    with $a>0$ the constant determined in Proposition \ref{prop: calS coercivity}. Additionally, the matrix
    \be
    \det \left( \begin{matrix}
        \left(-\wl, W_1 \right)_{L^2(\RR_+^2)} & \left( \partial_{x_1} \wl, W_1 \right)_{L^2(\RR_+^2)} \\
        \left(- \wl, W_2 \right)_{L^2(\RR_+^2)} & \left( \partial_{x_1} \wl, W_2 \right)_{L^2(\RR_+^2)}
    \end{matrix}
    \right) \not = 0.
    \label{nondegeneracy}
    \ee
    \end{corollary}

    \begin{proof} 
        For any $\ep_1>0$, by the density of $C_c^\infty(\RR_+^2)$ class in $L^2(\RR_+^2)$, there is $(W_1, W_2) \in C_c^\infty(\RR_+^2)$ such that
        \be
         \big \| W_1 - \partial_{x_1} \wl \big\|_{L^2(\RR_+^2)}  + \big\| W_2 - \mathbf{1}_B x_2 \big\|_{L^2(\RR_+^2)} < \ep_1.
         \label{W1W2: close}
        \ee
        We will show that with $\ep_1 > 0$ small enough, the $W_1, W_2$ chosen above will satisfy \eqref{spectral gap: modified} and \eqref{nondegeneracy} by a simple perturbation analysis. 
        
        For \eqref{spectral gap: modified}, we consider $h \in L^2(\RR_+^2)$ satisfying $h \perp W_i$ (i=1,2) and $\| h \|_{L^2(\RR_+^2)} =1$. Decompose it into
        \[  h = h^\perp + h^0 \in \left({\rm span} \{ \mathbf{1}_B x_2, \pa_{x_1} \wl \} \right)^\perp  \oplus  {\rm span} \{ \mathbf{1}_B x_2, \pa_{x_1} \wl \}.  \] 
        The $L^2$ closedness \eqref{W1W2: close} easily implies that for some $C > 0$, 
        \[ \| h^0 \|_{L^2(\RR_+^2)} \le C \ep_1, \quad \| h^\perp \|_{L^2(\RR_+^2)} \ge 1 - C\ep_1. \] 
        Together with the coercivity \eqref{spectral gap: version 1} and that $\tilde S$ is bounded on $L^2(\RR^2_+)$, this yields
    \begin{align*}
        \left( \tilde \calS h, h \right)_{L^2(\RR_+^2)}
        & = \left( \tilde \calS h^\perp, h^\perp \right)_{L^2(\RR_+^2)}
        + 2\left( \tilde \calS h^\perp, h^0 \right)_{L^2(\RR_+^2)} + \left( \tilde \calS h^0, h^0 \right)_{L^2(\RR_+^2)} \\
        & \ge a \| h^\perp \|_{L^2(\RR_+^2)}^2 + O(\ep_1) \ge a + O(\ep_1).  \end{align*}
        Hence \eqref{spectral gap: modified} holds when we choose $0< \ep_1 \ll a$. 
        
        For \eqref{nondegeneracy}, note that 
         \begin{align*}
       & \det \left(
        \begin{matrix}
            -\left(\wl, \partial_{x_1} \wl \right)  & (\partial_{x_1} \wl , \partial_{x_1} \wl) 
            \\
            -(\wl, \mathbf{1}_B x_2)  & \left( \partial_{x_1} \wl, \mathbf{1}_B x_2 \right)
        \end{matrix}
        \right)  \\
        = & \det \left(
        \begin{matrix}
            0  &  \| \pa_{x_1} \wl \|_{L^2}^2 
            \\
            -I[\wl]  & 0 \end{matrix}
        \right)
        = \| \partial_{x_1} \wl \|_{L^2(\RR_+^2)}^2 I[\wl] \not =0.
    \end{align*}
    Then with $\ep_1$ small enough, \eqref{W1W2: close} will imply \eqref{nondegeneracy}. 
    \end{proof}

\subsection{Coercivity of $\calQ$}
\label{sec 33}
In this section, we study the coercivity of $\calQ$ \eqref{def: Q} by obtaining coercivity in the exterior region and applying the previous interior coercivity Proposition \ref{prop: calS coercivity} and Corollary \ref{cor: calS coercivity}. 

\begin{proposition}[Coercivity of $\calQ$]
\label{prop: coercivity}
    There exists $\delta_0>0$ such that
    \be
    \calQ[h] \ge \delta_0 \| h \|_{L^2(\RR_+^2)}^2  + \frac{1}{2c_L^2} \int_{\RR_+^2} \errl(\mathbf x) h(\mathbf x) d\mathbf x, \text{ for all }  h \in \calK_{\delta_0},
    \label{coercivity of Q}
    \ee
    where the restriction set $\calK_{\delta_0}$ is defined by
    \begin{align}
        \calK_{\delta_0}:= 
    \Big\{ h \in L^2(\RR_+^2): 
    & \;\wl + h \ge 0 \text{ on $\RR_+^2$}, \; \text{h  is odd in $x_2$}, 
     \notag
     \\
     & x_2 h \in L^1(\RR_+^2), \;   \| h\|_{L^2}< \delta_0, \; 
    h \perp W_1, \text{ and } h\perp W_2  \Big\}.
    \label{restriction set}
    \end{align}
    with $W_1, W_2 \in C^\infty_c(\RR^2_+)$ from Corollary \ref{cor: calS coercivity}. 
    \end{proposition}

    \begin{proof}[Proof of Proposition \ref{prop: coercivity}] The result can be established by contradiction, and we will prove the result as follows.
   
    \mbox{}
    
   \noindent  \textbf{Step 1. Setup of contradiction.}
        We argue by contradiction and assume that there exists a sequence of functions $\{ h_n \}$ satisfying
\be
\omega_n := \wl + h_n \ge 0 \text{ on } \RR_+^2, \quad h_n \perp W_1, \quad h_n \perp W_2,
\label{sequence: prop 1}
\ee
and
\be
h_n \to 0\quad {\rm in}\,\,L^2(\RR^2_+) \quad {\rm as}\,\, n \to \infty, 
\label{sequence: prop 2}
\ee
such that
\be
\liminf_{n \to \infty} \frac{1}{\| h_n \|_{L^2}^2} \tilde \calQ[h_n] \le 0,
\label{sequence: prop 3}
\ee
where
\[
\tilde \calQ[h] = \left( \calS h ,h \right)_{L^2(\RR_+^2)} + \frac{1}{2c_L^2} \int_{\RR_+^2} \errl (\mathbf x) h(\mathbf x) d \mathbf x.
\]
In particular, from \eqref{sequence: prop 1}, we know that $h_n \ge 0$ on $\RR_+^2 \setminus B$. Additionally, if we define $g_n := \frac{1}{\| h_n \|_{L^2}} h_n$, then $\| g_n\|_{L^2} = 1$ and $ \frac{1}{\| h_n \|_{L^2}^2} \tilde \calQ[h_n]$ can be further written as
\[
 \liminf_{n \to \infty}\frac{1}{\| h_n \|_{L^2}^2} \tilde \calQ[h_n] =  \liminf_{n \to \infty} \left( (\calS g_n, g_n)_{L^2(\RR_+^2)} + \frac{1}{2c_L^2\| h_n \|_{L^2}}\int_{\RR_+^2} \errl(\mathbf x) g_n(\mathbf x) d\mathbf x \right) \le 0.
\]
Let us take a subsequence if necessary to obtain
\be
\lim_{n \to \infty} \left( (\calS g_n, g_n)_{L^2(\RR_+^2)} + \frac{1}{2c_L^2\| h_n \|_{L^2}}\int_{\RR_+^2} \errl(\mathbf x) g_n(\mathbf x) d\mathbf x \right) \le 0,
\label{contradi assum 2}
\ee
which possibly equals to $-\infty$.

\mbox{}

\noindent \textbf{Step 2. Vanishing of exterior integral $\int_{\RR_+^2} \errl(\mathbf x) g_n(\mathbf x) d \mathbf x$.} In this step, we will show 
\be
\lim_{n \to \infty} \int_{\RR_+^2} \errl(\mathbf x) g_n(\mathbf x) d \mathbf x = 0.
\label{limit: zero}
\ee
This will be the starting point for reducing the analysis to the interior operator $\tilde \calS$.

Using the interpolation inequality \eqref{energy ineq}, we can obtain that
\begin{align*}
    (\calS g_n, g_n)_{L^2(\RR_+^2)} & =  \frac{1}{2c_L^2} \| g_n\|_{L^2(\RR_+^2)}^2 - \frac{1}{2} \left(  \left( -\Delta_{\RR_+^2} \right)^{-1} g_n,g_n \right)_{L^2(\RR_+^2)} = \frac{1}{2c_L^2} \| g_n \|_{L^2(\RR_+^2)}^2  - E[g_n] \\
    & \ge \frac{1}{2c_L^2} \| g_n \|_{L^2(\RR_+^2)}^2 - C_* \| x_2 g_n \|_{L^2(\RR_+^2)} \| g_n \|_{L^2(\RR_+^2)}
    =\frac{1}{2c_L^2}  - C_* \| x_2 g_n \|_{L^2(\RR_+^2)}.
\end{align*}
From the definition of $\errl(\mathbf x)$ \eqref{error: def} and $g_n \ge 0$ in $B^c$, it is easy to check that
\begin{align*}
     \| x_2 g_n \mathbf{1}_{B_2^c} \|_{L^1(\RR_+^2)}
     = \int_{|\mathbf x| \ge 2, \; \mathbf x \in \RR_+^2} x_2 g_n(\mathbf x) d \mathbf x
     \le \frac{1}{2} \int_{|\mathbf x| \ge 2, \; \mathbf x \in \RR_+^2} \errl(\mathbf x) g_n(\mathbf x) d \mathbf x.
\end{align*}
Besides, by Cauchy-Schwartz inequality,
\begin{align*}
     \| x_2 g_n \mathbf{1}_{B_2} \|_{L^1(\RR_+^2)}
    \le \| x_2 \mathbf 1_{B_2} \|_{L^2(\RR_+^2)} \| g_n \|_{L^2(\RR_+^2)} \le C \| g_n \|_{L^2(\RR_+^2)},
\end{align*}
for some uniform constant $C =\| x_2 \mathbf 1_{B_2} \|_{L^2(\RR_+^2)} >0 $. Hence
\begin{align}
    & \quad \left( \calS g_n, g_n \right)_{L^2(\RR_+^2)} + \frac{1}{2c_L^2 \| h_n\|_{L^2}} \int_{\RR_+^2} \errl(\mathbf x) g_n(\mathbf x) d \mathbf x \notag\\ 
    & \ge \frac{1}{2c_L^2} - \frac{C_*}{2} \int_{\RR_+^2} \errl(\mathbf x) g_n(\mathbf x) d \mathbf x - CC_* +  \frac{1}{2c_L^2 \| h_n\|_{L^2}} \int_{\RR_+^2} \errl(\mathbf x) g_n(\mathbf x) d \mathbf x \notag \\
    & \ge \frac{1}{2c_L^2} - CC_* + \left( \frac{1}{2 c_L^2 \| h_n\|_{L^2}} - \frac{C_*}{2} \right) \int_{\RR_+^2} \errl(\mathbf x) g_n(\mathbf x) d \mathbf x ,
    \quad \text{ for all } n \ge 1.
    \label{contradi step 2}
\end{align}
Taking $N \gg 1$ so that $c_L^2 \| h_n \|_{L^2} \ll C_*^{-1}$ as $n \ge N$ from \eqref{sequence: prop 2}, combined with \eqref{contradi assum 2}, we conclude that there exists some constant $C_2 >0$ such that
\[
0 \le \frac{1}{4 c_L^2 \| h_n \|_{L^2}} \int_{\RR_+^2} \errl(\mathbf x) g_n(\mathbf x) d \mathbf x \le C_2, \text{ uniformly for } n \ge N.
\]
This implies \eqref{limit: zero} by using \eqref{sequence: prop 2} again.

\mbox{}

\noindent \textbf{Step 3. Vanishing of exterior energy.} In this step, we improve vanishing property \eqref{limit: zero} to the vanishing of exterior energy. More precisely, we claim that  
\be
E[g_n \mathbf{1}_{B^c}] \to 0, \text{ as } n \to \infty.
\label{vanish: exterior region}
\ee

Indeed, for any $\ep > 0$, we decompose this exterior function as 
\be 
g_n \mathbf{1}_{B^c} = g_{n; \ep}^{\rm ext} + g_{n; \ep}^{\rm mid},\quad {\rm where}\,\,g_{n; \ep}^{\rm ext} = g_n \mathbf{1}_{B_{1+\ep}^c},\quad g_{n; \ep}^{\rm mid} = g_n \mathbf{1}_{B_{1+\ep} \cap B^c}.
\ee 
Then applying the interpolation inequality \eqref{energy ineq},  we compute 
\[  E[g_{n; \ep}^{\rm mid}] \le C_* \| g_{n; \ep}^{\rm mid} \|_{L^2} \| x_2 g_{n; \ep}^{\rm mid}\|_{L^1} \le C_* \| x_2\|_{L^2(B_{1+\ep} \cap B^c)} \| g_n\|_{L^2}^2 \le C' \ep^\frac 12 \]
for some $C' > 0$; further combined with the lower bound of $\errl$ \eqref{error: def}, $g_n \mathbf{1}_{B^c}\ge 0$, and the vanishing property \eqref{limit: zero}, we obtain 
\begin{align*}
  E[g_{n; \ep}^{\rm ext}] \le C_* \| g_{n; \ep}^{\rm ext} \|_{L^2} \| x_2 g_{n; \ep}^{\rm ext}\|_{L^1} \le \frac{C_*}{1-(1+\e)^2}  \| g_n \|_{L^2} \int_{\RR^2_+} \errl(\mathbf x) g_n(\mathbf x) d \mathbf x = o_{n}(1). 
\end{align*}
Therefore, a Cauchy-Schwartz inequality yields that for any $\ep > 0$, we have 
\[ \limsup_{n \to \infty} E[g_n \mathbf{1}_{B^c}] \le 2 \limsup_{n \to \infty}\left( E[g_{n; \ep}^{\rm ext}] +E[g_{n; \ep}^{\rm mid}]\right) \le 2C' \ep^\frac 12.  \]
The vanishing of energy \eqref{vanish: exterior region} follows by taking $\ep \to 0$. 

\mbox{}

\noindent \textbf{Step 4. Conclusion of proof by interior coercivity.} 

Notice that \eqref{energy ineq} implies that $E[g_n \mathbf{1}_B] \le C \| x_2 \|_{L^2(B)}$ is uniformly bounded, we can compute by Cauchy-Schwartz and \eqref{vanish: exterior region} that 
\[ \left| E[g_n] - E[g_n \mathbf{1}_B] \right| \le \left(|E[g_n \mathbf{1}_B]| + |E[g_n \mathbf{1}_{B^c}]| \right)\cdot |E[g_n \mathbf{1}_{B^c}]|  = o_n(1).   \]
Thus recalling the definition of $\calS$ \eqref{second variation} and $\tilde \calS$ \eqref{linear op: cutoff version}, we compute 
\begin{align*}
    & \quad  \left( \calS g_n ,g_n \right)_{L^2(\RR_+^2)}
     = \frac{1}{2 c_L^2} \| g_n \|_{L^2(\RR_+^2)}^2 - E[g_n] \notag \\
    & = \frac{1}{2c_L^2} \|g_n \|_{L^2(\RR_+^2)}^2 - E[g_n \mathbf{1}_B] + o_n(1)  = \left( \tilde \calS g_n, g_n \right) + o_n(1) \ge a + o_n(1),
\end{align*}
where in the last inequality we applied the coercivity of $\tilde \calS$ in Corollary \ref{cor: calS coercivity} with $a > 0$. Noticing that $\int_{\RR_+^2} \errl(\mathbf x) g_n(\mathbf x) d\mathbf x \ge 0$, the lower bound of $ \left( \calS g_n ,g_n \right)_{L^2(\RR_+^2)}$ leads to a contradiction with our assumption \eqref{contradi assum 2}. Thus we have completed the proof.
    \end{proof}

\mbox{}

\subsection{Modulation}
\label{sec 34}
Based on the coercivity result established in Proposition \ref{prop: coercivity}, together with the modulation method, we are going to complete the proof of orbital stability of Lamb-Chaplygin dipole under sign condition and odd symmetry with a quantitative description. First of all, we will use the implicit function theorem to introduce the modulation parameters $(\alpha,\beta)$ to ensure the residue to satisfy certain orthogonal conditions.

\begin{proposition}[Set up of modulation analysis]
\label{thm: implicit function thm}
   There exists $\ep_{\rm mod} >0$, $C_{\rm mod}> 0$, and an open set $\Omega_{\rm mod} \subset \RR^2$ with $(0, 0) \in \Omega_{\rm mod}$ such that for any $\omega \in L^2(\RR_+^2)$ satisfying
   \[
   \| \omega - \wl \|_{L^2(\RR_+^2)} < \ep_{\rm mod},
   \]
   there exists uniquely $(\alpha, \beta) \in \Omega_{\rm mod}$ such that
   \be
     \omega(\cdot + \beta \bm e_1) - (1+\a) \wl  := h \perp W_1, W_2. \label{eqmodortho}
   \ee
   Moreover, the map $L^2 \owns \omega \mapsto (\alpha, \beta) \in \RR \times \RR$ is $C^1$, and we have 
      \be 
   |\a| + |\beta| \le C_{\rm mod} \| \omega - \wl \|_{L^2(\RR_+^2)}. \label{eqestalphabeta}
   \ee 
\end{proposition}

\begin{proof}
    We define a mapping $\calG: \RR \times \RR \times L^2 \to \RR^2$ as
    \begin{align*}
        \calG (\alpha, \beta, \omega)
        =  \left(
        \begin{matrix}
            \left( \omega - (1+ \alpha) \wl(\cdot -  \beta \mathbf{e}_1) , W_1 (\cdot -  \beta \mathbf{e}_1) \right)_{L^2(\RR_+^2)} \\
             \left( \omega - (1+ \alpha) \wl(\cdot -  \beta \mathbf{e}_1), W_2(\cdot -  \beta \mathbf{e}_1)\right)_{L^2(\RR_+^2)}
        \end{matrix}
        \right) 
    \end{align*}
    where $\calG (0,0,\wl) = \left( 0, 0 \right)^T$. For any $\omega$ close to $\wl$ in $L^2(\RR^2_+)$, we look for $(\a, \beta) = (\a[\omega], \beta[\omega]) $ such that $ \calG(\a, \beta, \omega) = 0$. 
    
  Firstly, we notice that $G(\alpha,\beta,\omega)$ is $C^1$ continuous in a small neighborhood of $(0,0, \wl) \in \RR \times \RR \times L^2(\RR^2_+)$, thanks to  $(W_1, W_2) \in \left( C_c^\infty(\RR_+^2) \right)^2$ and $\wl \in H^1(\RR^2)$. Next, we calculate the related Jacobian matrix at $(0,0,\wl)$, which is of form
    \begin{align*}
        \na_{\alpha,\beta} \calG (0,0, \wl)
        & =\left(
        \begin{matrix}
            -\left(\wl, W_1 \right)_{L^2(\RR_+^2)}  & (\partial_{x_1} \wl , W_1)_{L^2(\RR_+^2)}  \\
            -(\wl, W_2)_{L^2(\RR_+^2)}  & \left( \partial_{x_1} \wl, W_2\right)_{L^2(\RR_+^2)} 
        \end{matrix}
        \right).
    \end{align*}
    This matrix is not degenerate in view of \eqref{nondegeneracy}. 
    
    Then the existence and uniqueness of $(\a, \beta)$, the $C^1$ continuity and the estimate \eqref{eqestalphabeta} are a standard application of implicit function theorem in Banach space. 
\end{proof}

Our last preparation is the lower and upper bound of the distance functional $\bmrmd$ defined in \eqref{eqdefdist} by $L^2$ and $L^1_{x_2}$ norms.

\begin{lemma} \label{lem: dist} There exists $C_\bmrmd > 1$ such that for $h \in ( L^2 \cap L^1_{x_2})(\RR^2)$, 
\begin{align}
\| h \|_{L^2}^2 & \le  \bmrmd[h] \le \| x_2 h\|_{L^1} + \| h \|_{L^2}^2,
\label{eqdistcoer}\\
 \| x_2 h\|_{L^1} &\le C_\bmrmd \left( \bmrmd[h] + \sqrt{\bmrmd[h]} \right).
 \label{eqdistcoer2}
\end{align}
\end{lemma} 
\begin{proof}
The first estimate \eqref{eqdistcoer} is self-evident from the definition of $\bmrmd[h]$ and $\errl$ \eqref{error: def}. For second estimate \eqref{eqdistcoer2}, we exploit that 
\begin{align*} \| x_2 h \|_{L^1(B_2^c)}& \le 2 c_L^{-2} (\errl, h),\\
 \| x_2 h \|_{L^1(B_2)}& \le \| x_2 \|_{L^2(B_2)} \| h \|_{L^2} \le 2 \cdot \sqrt{4\pi}  \| h \|_{L^2} = 2\sqrt{\pi} \|h\|_{L^2}. 
\end{align*}
So \eqref{eqdistcoer2} holds with $C_\bmrmd = 2\sqrt\pi$.
\end{proof} 

Finally, we are in place to prove the main result Theorem \ref{thm: orbital stability sec3}.
\begin{proof}[Proof of Theorem \ref{thm: orbital stability sec3}]

For notational simplicity, denote 
\[ \ep := (\bmrmd [ \omega_0 - \wl])^\frac 12 \le \sqrt{\delta_1}. \] 
In particular, \eqref{eqdistcoer} implies that $\| \omega_0 - \wl \|_{L^2} \le \ep$. We denote $\omega$ as an admissible solution with initial data $\omega_0$ in the sense of Definition \ref{def: admissible solu}. 

We will first propagate $L^2$ estimates and modulation by bootstrap analysis and Lyapunov functional; meanwhile, we also prove \eqref{orbital stability: main result} and \eqref{smallness: parameters alpha}. Finally, we show \eqref{smallness: parameters beta} through modulation estimates. We stress that the constant $\delta_1$ will be required to be sufficiently small during the proof. 

\mbox{}

\noindent \textbf{Step 1. Set up of bootstrap analysis.}

Recall $\delta_0 >0$ from Proposition \ref{prop: coercivity}, $C_{\rm mod}>0$ from Proposition \ref{thm: implicit function thm}, and $C_\bmrmd>0$ from Lemma \ref{lem: dist}. We choose the coefficient $C_0, C_1 > 0$ in \eqref{orbital stability: main result} by
    \be
    \begin{split}
     C_0 &= 4 \max \left\{  \sqrt{\left( \min \left\{ \delta_0 c_L^2, 1/4 \right\}\right)^{-1} \left( c_L^2 + \frac{2C_\bmrmd^2}{I[\wl]} \right)} , 1 + C_{\rm mod} \| \wl \|_{L^2} \right \},\\ 
     C_1 &= 2 \max \left\{ \frac{C_\bmrmd}{I[\wl]} (2 + C_0), C_{\rm mod} \right\}.
     \end{split}
     \label{C0: choice}
    \ee
    We will finally pick $C_0 = (C_0')^2$ in Proposition \ref{prop: coercivity}.

    Now we define the exit time $T_{\max} \in [0,\infty]$ as 
    \be 
  T_{\max} := \sup\Big \{ T \ge 0: \exists  (\alpha(t), \beta(t)) \in \RR \times \RR\,\, \text {s.t.} \,\,\eqref{bootassump1}-\eqref{bootassump3}\,\, \text{hold for any} \, t \in [0,T] \Big \}.
    \label{Tmax: def}
    \ee 
    Here we define on $t \in [0,T_{\max}]$\footnote{For notational simplicity, we denote $[0, T_{\max}] = [0,\infty)$ when $T_{\max} = \infty$. }
    \be h(t, \cdot) := \omega(t, \cdot + \beta(t) \bm e_1) - (1+\a(t))\wl,\label{eqdefdecomph}
    \ee 
    and the bootstrap assumptions are 
    \begin{itemize}
        \item $L^2$-closedness: 
        \be 
    \| h(t) \|_{L^2} \le C_0 \ep.  \label{bootassump1} 
        \ee 
        \item Smallness of amplitude: 
        \be |\a(t)| \le C_1 \ep.  \label{bootassump2}
        \ee 
        \item Orthogonality: 
        \be h(t) \perp_{L^2} W_1, W_2,  \label{bootassump3} 
        \ee 
        where $W_1, W_2 \in C^\infty_c(\RR^2_+)$ are from Corollary \ref{cor: calS coercivity}. 
    \end{itemize}
    
    We now show $T_{\max} > 0$ using Proposition \ref{thm: implicit function thm}. Notice that $\| \omega_0 - \wl \|_{L^2} \le \ep$. With $\ep_{\rm mod} > 0$ from in Proposition \ref{thm: implicit function thm}, when $\delta_1 \ll \ep_{\rm mod}^2$ small enough, we can apply Proposition \ref{thm: implicit function thm} to obtain $\a(0), \beta(0)$ such that \eqref{bootassump3} holds at $t = 0$ and 
    \[ |\a(0)| \le C_{\rm mod} \ep,\quad \| h(0) \|_{L^2} 
    \le ( 1 + C_{\rm mod} \| \wl \|_{L^2} ) \ep. \] 
    In view of the choice of $C_0, C_1$ \eqref{C0: choice}, the bootstrap assumptions \eqref{bootassump1}-\eqref{bootassump2} at $t = 0$ are satisfied and even improved. Now the $L^2$-continuity of $\omega(t)$  (Definition \ref{def: admissible solu} (1)) and the continuity of $\omega \mapsto (\a, \beta)$ from Proposition \ref{thm: implicit function thm} enables us to construct $(\a(t), \beta(t))$ for some $[0, \tilde T]$ with $\tilde T > 0$, such that \eqref{bootassump1}-\eqref{bootassump3} hold. That verifies $T_{\max} \ge \tilde T > 0$. 
    
\mbox{}

\noindent \textbf{Step 2. Lyapunov functions.}

\noindent \textit{Step 2.1. Construction of Lyapunov functions.}

 Recall the Lagrangian functional $\calF$ from \eqref{functional: def} and the difference functional \eqref{def: Q}. We use the decomposition \eqref{eqdefdecomph} to compute 
\begin{align}
    &  \calF[\omega(t)] - \calF[\wl]  =  \calF[\omega(t, \cdot + \beta\bm e_1)] - \calF[\wl] = \calQ[\a(t)\wl + h(t)] \nonumber \\
    =& \calQ[h(t)] + \alpha(t)^2 \left( \calS \wl, \wl \right)_{L^2} + 2 \alpha(t) \Big( \calS \wl, h (t)\Big)_{L^2} \nonumber
    \\
    =& \calQ[h(t)]-\frac{1}{2}\alpha(t)^2 I[\wl] - \alpha(t) I[h(t)]
        + \frac{\alpha(t)}{c_L^2} \int_{\RR_+^2} \errl (\mathbf x) h(t,\mathbf x) d \mathbf x   ,  \label{eqdeficiency of Q}
\end{align}
where we exploited the $x_1$-translation invariance of $\calF$, and the algebraic relation \eqref{error: def} and $\tilde \calS \wl = -\frac 12 \mathbf{1}_B x_2$ from \eqref{F(0)}. We also compute the difference of impulse 
\begin{align*}
    I[\omega(t)] - I[\wl] = I[\omega(t,\cdot+\beta(t)\bm e_1)] - I[\wl] = \a(t) I[\wl] + I[h(t)],
\end{align*}
and its square
\[ (I[\omega(t)] - I[\wl])^2 = 2I[\wl] \left( \frac{1}{2} \alpha^2(t) I[\wl] + \alpha(t)I[h(t)] \right) + I[h(t)]^2. 
\]

Therefore, we introduce the Lyapunov functions on $[0, T_{\max}]$ as 
\begin{align}
    \calA(t) :=& \calF[\omega(t)] - \calF[\wl] + \frac{1}{2c_L^2 I[\wl]} \left( I[\omega(t)] - I[\wl] \right)^2  \label{eqdefcalAt} \\
    =& \calQ[h(t)] + \frac{\alpha(t)}{c_L^2} \int_{\RR_+^2} \errl (\mathbf x) h(t,\mathbf x) d \mathbf x  + \frac{1}{2c_L^2I[\wl]} (I[h(t)])^2,  \label{eqdefcalAt2} \\
\calB(t) :=& I[\omega(t)] - I[\wl].
\end{align}
Since $\omega(t,\cdot)$ satisfies the conservation laws (see Definition \ref{def: admissible solu} (3)), we have 
\[ \calA(t) = \calA(0),\quad \calB(t) = \calB(0),\quad \forall \,\,t \ge 0.\]

\mbox{}

\noindent \textit{Step 2.2. Upper bound at $t = 0$.}

Next, we claim the following upper bounds of $\calA(0), |\calB(0)|$ as
\be 
\calA(0) \le c_L^{-2} \left( c_L^2 + \frac{2C_\bmrmd^2}{I[\wl]} \right) \ep^2, \quad |\calB(0)| \le 2 C_\bmrmd \ep. \label{equpperLyap}\ee
Indeed, denoting $h_* = \omega_0 - \wl$, we can apply \eqref{eqdistcoer2} and $\bmrmd[h_*] = \ep < 1$ to compute
\be 
|\calB(0)| = |I[h_*]| \le \| x_2 h_* \|_{L^1} \le C_{\bmrmd}(\ep^2 + \ep) \le 2 C_\bmrmd \ep,
\ee 
and thereafter
\begin{align*}
\calA(0) =& \calQ[h_*] + \frac{1}{2c_L^2 I[\wl]} (I[h_*])^2 \\
\le& \frac{1}{c_L^2} \int_{\RR_+^2} \errl(\mathbf x) h_*(\mathbf x) d \mathbf x + \frac{1}{2c_L^2} \| h_* \|_{L^2(\RR_+^2)}^2 + \frac{1}{2c_L^2 I[\wl]} (I[h_*])^2 \\
\le&\left(1 + \frac{(2C_\bmrmd)^2}{2c_L^2 I[\wl]} \right) \ep^2.
\end{align*}

\mbox{}

\noindent \textit{Step 2.3. Lower bound on $t \in [0, T_{\max}]$.} 

Then, we use $\calA(t)$ and $\calB(t)$ to control $\bmrmd[h(t)]$ and $\a(t)$. We will prove the following estimates on $t \in [0, T_{\max}]$:
\begin{align}
    \calA(t) &\ge \min \left\{ \delta_0, (4c_L^2)^{-1} \right\} \bmrmd[h(t,\cdot)], \label{eqLyaplower1}\\
    |\a(t)| &\le (I[\wl])^{-1} \left( |\calB(t)| + C_\bmrmd (\bmrmd[h(t)] + \sqrt{\bmrmd[h(t)]} ) \right). \label{eqLyaplower2}
\end{align}

Indeed, \eqref{eqLyaplower2} easily follows from $\calB(t) = \a(t) I[\wl] + I[h(t)]$ and the estimate \eqref{eqdistcoer2}. For \eqref{eqLyaplower1}, we first notice that with $\delta_1 \ll 1$, the smallness of amplitude \eqref{bootassump2} implies $|\a(t)| \le \frac 14$. Plus the orthogonality condition \eqref{bootassump3} and the coercivity of $\calQ[h]$ from Proposition \ref{prop: coercivity}, we compute from the formula \eqref{eqdefcalAt2} that
\bee
 \calA(t) \ge  \delta_0 \| h \|_{L^2}^2 + \frac{1}{4 c_L^2} \int_{\RR_+^2} \errl (\mathbf x) h(t, \mathbf{x}) d \mathbf x \ge \min \left\{ \delta_0, (4c_L^2)^{-1} \right\} \bmrmd[h(t)].
\eee 
Here the last inequality exploits $h\big|_{B^c \cap \RR_+^2} = \omega(t, \cdot + \beta \bm e_1)\big|_{B^c \cap \RR_+^2}  \ge 0$, so that $\int_{\RR_+^2} \errl (\mathbf x) h(t, \mathbf{x}) d \mathbf x = \| \errl h(t,\cdot) \|_{L^1(\RR_+^2)}$.

\mbox{}

\noindent \textit{Step 2.4. Uniform boundedness of $\bmrmd[h(t)]$ and $|\a(t)|$.} 

Finally, we combine the conservation laws and estimates of $\calA(t)$ and $\calB(t)$ from above to bound $\bmrmd[h(t)]$ and $|\a(t)|$ on $t \in [0, T_{\max}]$. Recall the definition of $C_0, C_1$ in \eqref{C0: choice}. From $\calA(t) = \calA(0)$ and its estimates \eqref{equpperLyap} and \eqref{eqLyaplower1}, we obtain 
\be 
 \bmrmd[h(t)] \le \left(\min \left\{ \delta_0, (4c_L^2)^{-1} \right\}\right)^{-1}  c_L^{-2} \left( c_L^2 + \frac{2C_\bmrmd^2}{I[\wl]} \right) \ep^2 \le \frac{C_0^2}{4} \ep^2,\quad \forall \,t \in [0, T_{\max}].  \label{equniestboot1}
\ee 
Further using $\calB(t) = \calB(0)$ and its estimates \eqref{equpperLyap} and \eqref{eqLyaplower2}, we have 
\be
 |\a(t)| \le (I[\wl])^{-1} \left( 2C_\bmrmd + C_\bmrmd \left( \frac{C_0^2 \e}{4} + \frac{C_0}{2} \right) \right) \ep \le \frac{C_1}{2} \ep, \quad \forall \,t \in [0, T_{\max}].  \label{equniestboot2} 
\ee
Here we required $\delta_1 \ll 1$ such that $C_0 \ep \le C_0\sqrt{ \delta_1} \le 1$. In particular, \eqref{eqdistcoer} and \eqref{equniestboot1} implies that 
\be 
\| h(t) \|_{L^2} \le \frac{C_0}{2} \ep,\quad \forall \,t \in [0, T_{\max}].  \label{equniestboot3} 
\ee

\mbox{}

\noindent \textbf{Step 3. Propagation of bootstrap assumption and proof of  \eqref{orbital stability: main result}, \eqref{smallness: parameters alpha}.}  

In this step, we will prove $T_{\max} = \infty$ by contradiction. As a corollary, the uniform bounds \eqref{equniestboot1}-\eqref{equniestboot2} now hold on $[0, \infty)$, verifying \eqref{orbital stability: main result}, \eqref{smallness: parameters alpha} with $C_{\rm stab} = C_0^2$. We also obtain continuity of $t \mapsto (\a(t), \beta(t))$ since they are determined locally by Proposition \ref{thm: implicit function thm}. 

Now assume $T_{\max} < \infty$. We will construct $(\a(t), \beta(t))$ with bootstrap assumptions holding for $t \in [T_{\rm max}, T_{\rm max} + \tilde T]$, where $\tilde T > 0$. This contradicts the definition of $T_{\rm max}$. Indeed, the uniform bounds \eqref{equniestboot2}-\eqref{equniestboot3} improve the bootstrap assumption \eqref{bootassump1}-\eqref{bootassump2} at $t = T_{\max}$. With $\delta_1 \ll 1$ small enough, this implies 
\[ \| \omega(T_{\max},\cdot + \beta(T_{\max}) \bm e_1) - \wl \|_{L^2} \le \frac 12 (C_0 + C_1 \| \wl \|_{L^2}) \ep \le \frac 12 \ep_{\rm mod}.\]
Applying Proposition \ref{thm: implicit function thm} to $\omega(T_{\max},\cdot + \beta(T_{\max}) \bm e_1)$, there exists $(\tilde \a(T_{\max}), \tilde \beta(T_{\max})) \in \Omega_{\rm mod}$ such that $\tilde h :=  \omega(T_{\max},\cdot + (\tilde \beta (T_{\max})+ \beta(T_{\max})) \bm e_1) - (1 + \tilde \a(T_{\max}))\wl$ is orthogonal to $W_1, W_2$. Since $(\a(T_{\max}), 0) \in \Omega_{\rm mod}$ from $|\a(T_{\max})| \le \frac 12 C_1 \ep$ when $\delta_1 \ll 1$, the uniqueness statement in Proposition \ref{thm: implicit function thm} indicates that $(\tilde \a(T_{\max}), \tilde \beta(T_{\max})) = (\a(T_{\max}), 0)$. 

Thereafter, similar to the proof of $T_{\max} > 0$ in Step 1, we can exploit the $L^2$-continuity of $\omega(t,\cdot + \beta(T_{\max})\bm e_1)$ (Definition \ref{def: admissible solu} (1)) and the continuity of $\omega \mapsto (\a, \beta)$ from Proposition \ref{thm: implicit function thm} enables us to continuously extend $(\a(t), \beta(t))$ to slightly beyond $T_{\max}$ with \eqref{bootassump1}-\eqref{bootassump3} holding true. That contradicts the definition of $T_{\max}$ and concludes $T_{\max} = \infty$ as well as  \eqref{orbital stability: main result}, \eqref{smallness: parameters alpha}.

\mbox{}

\noindent \textbf{Step 4. Proof of \eqref{smallness: parameters beta}.}

Lastly, we prove the derivative estimate of $\alpha$ and $\beta$. By choosing the test function $\varphi(t, \mathbf{x}) = \eta(t) \psi(\mathbf x)$ in \eqref{weak sol: def}, we obtain the time-wise equation 
\be \frac{d}{dt} (\omega(t), \psi) = \int_{\RR^2} \omega(t)  \mathbf{u}(t) \cdot \nabla \psi dx,\quad \mathbf{u} := \nabla^\perp \Delta^{-1} \omega.
  \label{weak sol: def2}
  \ee 
in the distributional sense for any test function $\psi \in C_c^\infty(\RR^2)$.
Plugging in the decomposition \eqref{eqdefdecomph} into \eqref{weak sol: def2} yields
\be \begin{split}
    &\frac{d}{dt} \left( h, \varphi \right) + (\beta'-1) 
   \Big[  (-\partial_{x_1}\wl, \varphi) + \left( h+ \alpha \wl, \partial_{x_1} \varphi \right) \Big] + \alpha'(\wl, \varphi) 
   \\ 
   =& \left( (h+\a \wl) \nabla^\perp \Delta^{-1}\wl ,\nabla \varphi \right) + \left( \wl \nabla^\perp \Delta^{-1}(h+\a \wl) ,\nabla \varphi \right) \\
   -& \left(h+\a \wl, \pa_{x_1} \varphi\right) + \left( (h + \a \wl) \nabla^\perp \Delta^{-1}(h+\a \wl) ,\nabla \varphi\right),
\end{split} \label{eqweakeqh}
\ee
for any test function $\varphi \in C_c^\infty(\RR^2)$.
Here we exploited the equation of $\wl$:
\[
\nabla \wl \cdot \nabla^\perp \Delta^{-1} \wl = \pa_{x_1} \wl.
\]
From the estimates \eqref{equniestboot1} and \eqref{equniestboot2}, we estimate 
\bee
  \| h + \a \wl \|_{L^2} \le \left(C_0 + C_1 \| \wl \|_{L^2}\right) \ep.
  \eee
Moreover, recalling \eqref{energy ineq}, Lemma \ref{lem: dist}, \eqref{equniestboot1} and \eqref{equniestboot2}, there exists $\tilde C_*>0$ such that
\begin{align*}
& \quad \| \na^\perp \Delta^{-1} (h+ \a \wl) \|_{L^2(\RR_+^2)}^2
= 2 E[h+\a \wl]  \\
& \le 2C_* \| h + \a  \wl \|_{L_{x_2}^1(\RR_+^2)} \| h + \a \wl \|_{L^2(\RR_+^2)}  \\
& \le \tilde C_* \left( C_0^2 + C_1^2 \| \wl \|_{L^2 \cap L_{x_2}^1(\RR_+^2)}^2 \right) \ep^2.
\end{align*}
Since $\omega(t) \in C^0([0,\infty), L^2(\RR^2) \cap L_{x_2}^1(\RR^2))$ from Definition \ref{def: admissible solu} (1) and $(\a(t), \beta(t)) \in C^0_{loc}(\RR)$ from Step 3, we thus obtain continuity of $t \mapsto h + \a \wl$ and  $t \mapsto \nabla^\perp \Delta^{-1} (h + \a \wl)$ in $L^2(\RR^2)$.

Therefore, we choose the test functions in \eqref{eqweakeqh} as $W_1$ and $W_2$ respectively. Using that $W_1, W_2 \in C_c^\infty(\RR^2)$, we can easily bound all inner products involving $h$ in \eqref{eqweakeqh} by $O\left(\|\nabla W_i \|_{L^2 \cap L^\infty}\cdot \ep \right)$, and moreover they are continuous in time. Together with the fact that $(h(t,\cdot),W_i) \equiv 0$ ($i=1,2$) for any $t \ge 0$, we obtain that 
   \begin{align*}
   \begin{cases}
        -(\beta'-1) \left( (\partial_{x_1} \wl , W_1) + O(\ep)\right) + \alpha' (\wl
       , W_1) =O(\ep), \\
       -(\beta'-1) \left( (\partial_{x_1} \wl , W_2) + O(\ep)\right) + \alpha' (\wl
       , W_2) = O(\ep).
   \end{cases}
   \end{align*}
   Combining with the nondegeneracy given in \eqref{nondegeneracy} and the choice of $0< \delta_1 \ll 1$ small enough, we can invert the coefficient matrix on the left to obtain $(\a',\beta') \in C^0(\RR)$ and that for some $C_2 > 0$, 
   \[
   |\beta'(t)-1| + |\alpha'(t)| \le C_2 \ep, \quad  \text{for any } t \ge 0.
   \]
   This concludes the proof of \eqref{smallness: parameters beta} and thus Theorem \ref{thm: orbital stability sec3}.
\end{proof}

\bibliographystyle{plain}
\bibliography{Bib-1}

\end{document}